\documentclass{amsart}
\usepackage[marginratio=1:1,height=584pt,width=390pt,tmargin=117pt]{geometry}
\usepackage{graphicx}
\usepackage{amssymb}
\usepackage{amsmath}
\usepackage{graphicx}
\usepackage{float}
\usepackage{placeins}
\usepackage{xargs}  
\usepackage[pdftex,dvipsnames]{xcolor} 
\usepackage{array}
\usepackage{calc}
\usepackage[ruled,noline]{algorithm2e}
\SetKwInOut{Input}{Input} 
\SetKwInOut{Output}{Output}
\usepackage[font=small,labelfont=bf]{caption}
\usepackage{subcaption}
\usepackage{rotating}
\usepackage[outdir=./]{epstopdf}

\newcolumntype{?}{!{\vrule width 1.5pt}}

\usepackage[colorlinks=true,
           linkcolor=red,
            urlcolor=blue,
            citecolor=black]{hyperref}

\numberwithin{equation}{section}

\theoremstyle{plain}
\newtheorem{thm}{Theorem}[section]
\newtheorem{lem}[thm]{Lemma}
\newtheorem{prop}[thm]{Proposition}

\theoremstyle{definition}
\newtheorem{defn}{Definition}[section]

\theoremstyle{remark}
\newtheorem*{rem}{Remark}

\newcommand{\wh}[1]{\widehat{#1}}

\newcommand{\RR}{\mathbb{R}}

\newcommand{\NN}{\mathbb{N}}

\renewcommand{\Re}{\operatorname{Re}}             

\newcommand{\Laplacian}{\Delta}
\newcommand{\dds}{\text{d}s}
\newcommand{\figref}[1]{\figurename~\ref{#1}}

\newcommand*{\arr}{\longrightarrow}

\author{Andrea Scapin}
\address{Department of Mathematics, 
ETH Z\"urich, 
R\"amistrasse 101, CH-8092 Z\"urich, Switzerland.}
\email{andrea.scapin@sam.math.ethz.ch}
\title[Electro-sensing of inhomogeneous targets]{Electro-sensing of inhomogeneous targets}
\thanks{\footnotesize  
This work was supported by the SNF grant 200021-172483.}
\subjclass[2010]{35R30,35J05,31B10,35C20,78A30}

\keywords{weakly electric fish, electro-sensing, shape classification, inhomogeneous target}

\begin{document}
\maketitle

\begin{abstract} This paper addresses the electro-sensing problem for weakly electric fish in the case of inhomogeneous targets. 
It aims at providing a shape descriptor-based classification for inhomogeneous targets from measurements of the potentials on the skin of the fish.  The approach is based on new invariants for the contracted generalized polarization tensors associated with inhomogeneous objects. The numerical simulations show that by comparing these invariants with those in a dictionary of precomputed homogeneous and inhomogeneous targets, one can successfully classify the inhomogeneous target. 
\end{abstract}

\section{Introduction}

Electric fish orient themselves at night in complete darkness by using their active electro-sensing system. They generate a stable, relatively high-frequency, weak electric field and perceive the transdermal potential modulations caused by nearby targets with different electromagnetic properties than the surrounding water \cite{p12,p13,p14}. Since they have an electric sense that allows underwater navigation, target classification and intraspecific communication, they are privileged animals for bio-inspiring man-built autonomous systems \cite{caputi}. In fact, active electro-sensing has motivated an increasing number of
experimental, behavioral, biological, and computational studies
since Lissmann and Machin’s work \cite{p5,p6,p7,p8,p9,p10,p11,p12}. The growing interest in electro-sensing could be
explained not only by the curiosity of discovering a sixth sense,
electric perception, that is not among one of  our own senses, but
also by potential bio-inspired applications in underwater autonomous robotics.
It is challenging to equip robots with electric perception and
provide them, by mimicking weakly electric fish, with navigation, imaging
and classification capabilities in dark or turbid environments
\cite{p14,p15}.

Mathematically speaking, the electro-sensing problem is to detect and locate the target
and to identify its shape and material parameters given the current
distribution over the skin of the fish. In other words, electro-sensing system performance can be assessed with respect to four fundamental tasks: target detection, estimation of a target's location, shape and internal structure. Due to the fundamental ill-posedness
of this imaging problem, it is very intriguing to see how much
information weakly electric fish are able to recover. The electric
field perturbation due to the target is a complicated highly
nonlinear function of its shape, electromagnetic parameters, and distance from
the fish. Thus, understanding analytically this electric sensing is
likely to give us insight in this regard \cite{p5,p6,p7, p9}. 
A simple physical model for the electric responses of the polarised targets has been proposed in \cite{p11}. The model  shows that the target's position and size are intricately related in the measurements of the trans-cutaneous currents projected onto the skin. 
Numerical approaches have also been driven for simplified geometries using a finite differences scheme in \cite{helli}, a finite elements method in \cite{hoshi}, and a boundary element method approach in \cite{p5}. The geometry of the fish is simplified by an ellipse and is divided into two areas: the thin skin with low conductivity and the interior of the body. The target's shape is a disk. 
In this simple model, the electric images projected onto the fish skin are fundamentally blurry and difficult to interpret \cite{p12,p13}. In \cite{hoshi}, the skin's and body's conductivity values are optimized in order to approximate as well as possible the experimentally measured field. The result is that the optimal conductivity is not uniform, being higher in the tail region. 
 

More recently, in \cite{Am2}, a rigorous model for the electro-location of a target
around the fish has been derived. Using the fact that the electric
current produced by the electric organ is time harmonic with
a known fundamental frequency, a space-frequency location
search algorithm has been introduced. Its robustness with respect to
measurement noise and its sensitivity with respect to the number
of frequencies, the number of sensors, and the distance to the
target have been illustrated. In the case of disk- and ellipse-shaped
targets, it has been shown that the conductivity, the permittivity, and the size of the
targets can be reconstructed separately from multifrequency
measurements. In \cite{maciver}, a capacitive sensing method has recently been implemented. It has been shown that the size of a capacitive sphere can be estimated from multifrequency electrosensory data.  In \cite{faouzi}, uniqueness and
stability estimates to the considered electro-sensing inverse problem have been established.

There are still many longstanding problems in electro-sensing. In particular, shape identification and classification are considered to be
the most challenging ones.  They have yet to be analyzed and understood.  In \cite{Am3,Am5}, two schemes that allow one to recognize and classify targets from
 measurements of the electric field perturbations
induced by the targets have been presented and analyzed.  The first algorithm is based on shape descriptors
for nonbiological targets and the second one is based on spectral induced
polarizations that can be used to image living biological
targets, which have frequency-dependent electromagnetic 
parameters due to the capacitive effects induced by their cell membrane structures \cite{jmpa}. In  \cite{Am3}, one first extracts the generalized (or high-order) polarization tensors of the target from the data. These tensors, first introduced in \cite{Am4}, are intrinsic
geometric quantities and constitute the right class of features to
represent the target shapes \cite{Am3}. The shape features are
encoded in the polarization tensors. The extraction of the generalized polarization tensors
can be achieved by a least-squares method. The noise level in the
reconstructed generalized polarization tensors depends on the angle of view. The larger the
angle of view, the more stable the reconstruction. Then from the extracted features one computes the invariants under rigid motion and scaling.  Comparing these invariants with those in a dictionary
of precomputed shapes, one can successfully classify the nonbiological
target. 

Since more complex objects may have arbitrary shapes and multiple layers of different dielectric materials warranting,  a deeper analysis of the full response is required. It is the objective of the present paper to extend the approach proposed in  \cite{Am3} to inhomogeneous targets.  Let us now recall the model of electro-sensing derived in \cite{Am2}: the body of the fish is $\Omega$, an open bounded set in $\RR^2$, with smooth boundary $\partial \Omega$, and with outward normal unit vector denoted by $\nu$. 
The electric organ is a dipole $f(x)$ inside $\Omega$ or a sum of point sources inside $\Omega$ satisfying the charge neutrality condition. The skin of the fish is very thin and highly resistive. Its effective thickness, that is, the skin thickness times the contrast between the water and the skin conductivities, is denoted by $\xi$, and it is much smaller than the fish size. We assume that the conductivity of the background medium is one. We consider a smooth bounded target $D= z + \delta B$, where $z$ is its location, and $B$ is a smooth bounded domain containing the origin. We assume that the conductivity of $D$ is a scalar function $\gamma(x) \neq 1$ for $x\in D$ with $\gamma(x) = \wh{\gamma}((x-z)/\delta)$. Also, let $\gamma \in L^{\infty}(\RR^2)$ satisfy the uniform ellipticity condition that for some $\lambda > 0$, $\lambda^{-1} \le \gamma \le \lambda$. In the presence of $D$, the electric potential emitted by the fish is the solution to the following equations: 
\begin{equation} \label{eq:model_u} \begin{cases} \Delta u = f  & \mbox{in } \Omega ,\\ \nabla \cdot (1 + (\gamma-1) \chi_D ) \nabla u = 0  &  \mbox{in } \RR^2 \setminus \overline{\Omega} ,\\ u|_+ - u |_- = \xi \dfrac{\partial u}{\partial \nu} \biggr |_+  &  \mbox{on } \partial \Omega,\\  \dfrac{\partial u}{\partial \nu} \biggr |_- = 0   & \mbox{on } \partial \Omega ,\\ |u(x)| = O(|x|^{-1})  & \mbox{as } |x| \to \infty.\end{cases} \end{equation}
Here, $\chi_D$ is the characteristic function of $D$, $\partial/\partial \nu$ is the normal derivative, and $|_{\pm}$ denotes the limits from, respectively, outside and inside $\Omega$. The static background potential $U$, i.e., the electric potential without any target,  
 is the unique solution to (\ref{eq:model_u}) with a constant conductivity equal to $1$ outside the body of the fish $\Omega$:
\begin{equation} \label{eq:model_U} \begin{cases} \Delta U = f  & \mbox{in } \Omega , \\ \Delta U = 0 &  \mbox{in } \RR^2 \setminus \overline{\Omega} ,\\ U|_+ - U |_- = \xi \dfrac{\partial U}{\partial \nu} \biggr |_+   &  \mbox{on } \partial \Omega ,\\  \dfrac{\partial U}{\partial \nu} \biggr |_- = 0   & \mbox{on } \partial \Omega ,\\ |U(x)| = O(|x|^{-1})  & \mbox{as } |x| \to \infty.\end{cases} \end{equation}
A dipole approximation for small homogeneous targets away from the fish has been derived in \cite{Am2}. It is given in terms of the generalized polarization tensors (GPTs). The concept of GPTs for an inhomogeneous target has been first considered in \cite{Am1}. However in \cite{Am1}, a  model much simpler then the weakly electric fish has been taken into account. The aim of the present paper is to extend the notion of generalized polarization tensors (GPTs) to the fish model described above and to introduce an efficient shape descriptor-based classification for inhomogeneous targets from measurements of the potentials on the skin of the fish.  The approach is based on new invariants for particular linear combinations of the GPTs associated with inhomogeneous objects.

The paper is organized as follows. In Section \ref{sec1}, we  derive a boundary integral representation for the perturbation of the potential due to the presence of the target. We introduce the GPTs associated with the inhomogeneous target $D$ as the building blocks of the multipolar asymptotic expansion of the boundary measurements of $u|_+$ on $\partial \Omega$ in terms of the size of $D$. In Section \ref{sec2}, we consider a particular linear combination of the GPTs, called contracted generalized polarization tensors (CGPTs), and generalize the translation, rotation and scaling formulas  for the contracted GPTs associated with homogeneous targets first derived in \cite{Am3} to those associated with the inhomogeneous target $D$. Based on such formulas, we build transform invariants for the CGPTs and propose a matching algorithm for retrieving inhomogeneous targets. In  Section \ref{sec3}, we present a variety of numerical simulations to illustrate the performance of the proposed matching algorithm.  We aim at recognizing a specific inhomogeneous target by means of a dictionary-matching approach. The considered dictionary of targets contains both homogeneous and inhomogeneous objects. The latters are obtained by inserting inside the homogeneous targets inclusions with different conductivities. Similarily to what has been done in \cite{Am3}, the numerical simulations we perform confirm  that extracting generalized polarization tensors of an inhomogeneous target from the  data and comparing invariants with those of learned elements in a dictionary yields a classification procedure with a good performance in the full-view case and with small measurement noise level.

\section{CGPTs for the weakly electric fish model} \label{sec1}

\subsection{Boundary integral representation}

The two-dimensional model we want to study is \eqref{eq:model_u} where the target $D$ is assumed to be inhomogeneous.

\medskip

First, we recall a boundary integral representation for the perturbation of the potential, namely $u - U$, where $U$ and $u$ are solutions to \eqref{eq:model_U} and \eqref{eq:model_u} respectively. Let $\Gamma_R$ be the Green's function associated with Robin boundary conditions, that is defined for $x \in \RR^2 \setminus \overline{\Omega}$ by

\begin{equation} \label{eq:model_GR} \begin{cases} - \Delta_y \Gamma_R(x,y) = \delta_x(y) , & y \in \RR^2 \setminus \overline{\Omega}, \\ \Gamma_R (x,y) |_+ -  \xi \dfrac{\partial \Gamma_R}{\partial \nu_x} (x,y) \biggr |_+ = 0  , & y \in \partial \Omega , \\  \left |\Gamma_R(x,y) + \frac{1}{2 \pi} \log |y| \right | = O(|y|^{-1}) , & |y| \to \infty.\end{cases} \end{equation}

We consider the divergence-type equation in \eqref{eq:model_u} posed on $\RR^2 \setminus \overline{\Omega}$ and test it with the solution  $\Gamma_R$ to \eqref{eq:model_GR}:
\begin{equation} \label{eq:test_the_equation_GammaR}\Gamma_R(x,y)\, \nabla_y \cdot \gamma (y) \nabla_y u(y)  = 0.\end{equation}
Integrating \eqref{eq:test_the_equation_GammaR} over $B_R \setminus (\overline{\Omega} \cup \overline{D})$, we get
\[ \int_{B_R \setminus (\overline{\Omega} \cup \overline{D})} \Gamma_R(x,y)  \nabla \cdot \gamma (y) \nabla_y u(y)  \,  \text{d}y =  \int_{B_R  \setminus (\overline{\Omega} \cup \overline{D})} \Delta_y u(y) \,\Gamma_R(x,y)\text{d} y = 0 . \]
Applying Green's theorem we obtain
\[ \begin{split} 0 &= \int_{B_R  \setminus (\overline{\Omega} \cup \overline{D})} \Delta_y u(y) \,\Gamma_R(x,y)\text{d} y \\ & = \int_{B_R \setminus (\overline{\Omega} \cup \overline{D})} u(y) \cdot (- \Delta_y \Gamma_R) (x,y)\, \text{d}y + \int_{\partial (B_R \setminus (\overline{\Omega} \cup \overline{D}) )}  \left ( \frac{\partial u}{\partial \nu_y} (y)\; \Gamma_R(x,y)  -  \frac{\partial \Gamma_R}{\partial \nu_y} \; u(y)\,  \right )  \text{d} s_y \\ & =  u(x) + \int_{\partial (B_R \setminus (\overline{\Omega} \cup \overline{D}) )}  \left ( \frac{\partial u}{\partial \nu}\; \Gamma_R  -  \frac{\partial \Gamma_R}{\partial \nu_y} \; u\, \right ) \text{d} s .  \end{split}  \]

So, we have
\[ \begin{split} u(x) &= \int_{\partial (B_R \setminus (\overline{\Omega} \cup \overline{D}) )}  \left ( \frac{\partial \Gamma_R}{\partial \nu_y} \; u -   \frac{\partial u}{\partial \nu}\; \Gamma_R \right )    \text{d} s\\ & = \int_{\partial B_R}   \left (\frac{\partial \Gamma_R}{\partial \nu_y} \; u -   \frac{\partial u}{\partial \nu}\; \Gamma_R \right )   \text{d} s - \int_{\partial \Omega }  \left ( \frac{\partial \Gamma_R}{\partial \nu_x} \; u -   \frac{\partial u}{\partial \nu}\; \Gamma_R \right )   \text{d} s   -   \int_{\partial D} \left (  \frac{\partial \Gamma_R}{\partial \nu_y} \; u -   \frac{\partial u}{\partial \nu}\; \Gamma_R \right )   \text{d} s . \end{split}  \]

Let $U$ be the static background solution defined in \eqref{eq:model_U}. Then
\[ \Delta U = 0  \quad\qquad \text{ in } \RR^2 \setminus \overline{\Omega} . \]
Multiplying by $\Gamma_R$ and integrating over $B_R \setminus \overline{\Omega}$ we get
\[ U(x) = \int_{\partial B_R} \left (  \frac{\partial \Gamma_R}{\partial \nu_y} \; U -   \frac{\partial U}{\partial \nu}\; \Gamma_R \right ) \,\text{d} s  -  \int_{\partial \Omega} \left (  \frac{\partial \Gamma_R}{\partial \nu_y} \; U -   \frac{\partial U}{\partial \nu}\; \Gamma_R \right ) \text{d} s . \]

Omitting the contributions of the integrals on $\partial B_R$, that are negligible as $R \to +\infty$, we have that

\[ \begin{split} u - U  &=  \int_{\partial \Omega} \left (  \frac{\partial \Gamma_R}{\partial \nu_y} \; U -   \frac{\partial U}{\partial \nu}\; \Gamma_R \right )   - \int_{\partial \Omega }  \left ( \frac{\partial \Gamma_R}{\partial \nu_y} \; u -   \frac{\partial u}{\partial \nu}\; \Gamma_R \right ) \,  \text{d} s  -   \int_{\partial D} \left (  \frac{\partial \Gamma_R}{\partial \nu_y} \; u -   \frac{\partial u}{\partial \nu}\; \Gamma_R \right )   \text{d} s \\ & = \int_{\partial \Omega} \frac{\partial \Gamma_R}{\partial \nu_y} \left (  U - u \right )    -  \Gamma_R \left ( \frac{\partial U}{\partial \nu} -   \frac{\partial u}{\partial \nu} \right ) \text{d} s  -   \int_{\partial D} \left (  \frac{\partial \Gamma_R}{\partial \nu_y} \; u -   \frac{\partial u}{\partial \nu}\; \Gamma_R \right )  \text{d} s . \end{split}\]
Then, using the Robin's boundary conditions for $\Gamma_R , U$ and $u$, we obtain that
\[ = \int_{\partial \Omega}  \frac{\partial \Gamma_R}{\partial \nu_y}   \left ( U_{|+} - u_{|+} \right )    -  \xi  \frac{\partial \Gamma_R}{\partial \nu_y}  \left ( \frac{U_{|+} - U_{|-}}{\xi} -  \frac{u_{|+} - u _{|-}}{\xi} \right ) \text{d} s  -   \int_{\partial D} \left (  \frac{\partial \Gamma_R}{\partial \nu_y} \; u -   \frac{\partial u}{\partial \nu}\; \Gamma_R \right )   \text{d} s . \]
Since $u_{|-} = U_{|-}$, we get
\[ = \int_{\partial \Omega}  {\frac{\partial \Gamma_R}{\partial \nu_y}   \left ( U_{|+} - u_{|+} \right )}   -   {\frac{\partial \Gamma_R}{\partial \nu_y}  \left ( U_{|+} - u_{|+} \right ) } \text{d} s   -   \int_{\partial D} \left (  \frac{\partial \Gamma_R}{\partial \nu_y} \; u_{|+} -   \frac{\partial u}{\partial \nu} \biggr |_{+}\; \Gamma_R \right )   \text{d} s  \, . \]
Therefore, 
\begin{equation} \label{eq:bdy_int_rep_1} {(u - U)(x) =  \int_{\partial D}  \left (   \frac{\partial u}{\partial \nu} \biggr |_{+}(y) \; \Gamma_R (x,y)- \frac{\partial \Gamma_R}{\partial \nu_x}(x,y) \; u_{|+}(y) \, \right )  \text{d} s_y } \, .\end{equation}

Let us now define the Neumann-to-Dirichlet (NtD) map $$\Lambda_{\gamma} \left [ \gamma \frac{\partial u }{\partial \nu} { \bigg |_-}  \right ] = u_{| \partial D}.$$
The transmission condition on $\partial D$ $$\gamma \frac{\partial u }{\partial \nu} \bigg |_- = \frac{\partial u }{\partial \nu} \bigg |_+$$ yields
\[ (u - U) (x)=  \int_{\partial D}  g(y)  \; \Gamma_R(x,y) \,\text{d} s_y -  \int_{\partial D} \frac{\partial \Gamma_R}{\partial \nu_y}(x,y) \; \Lambda_{\gamma} [g] (y)  \,  \text{d} s_y\,, \]
with $g= {\partial u}/{\partial \nu} |_{+}$. 
\medskip

For $x \in \RR^2 \setminus ( \overline{\Omega} \cup \overline{D} )$,  
\[ \Lambda_1 \left ( \frac{\partial \Gamma_R (x, \cdot  )}{\partial \nu_y} \right ) = \Gamma_R (x, \cdot ) -  \frac{1}{| \partial D|} \int_{\partial D} \Gamma_R(x,y) \text{ d} s _y  \qquad \text{ on } \partial D , \]
where $\Lambda_1 = \Lambda_{\gamma\equiv 1}$, 
and hence, 
\[   \frac{\partial \Gamma_R (x , \cdot )}{\partial \nu_y} = \Lambda_1^{-1} [\Gamma_R ( x, \cdot)]  \qquad \text{ on } \partial D, \] since $\Lambda_1: H_0^{- 1/2}(\partial D) \to H_0^{1/2}(\partial D)$ is invertible. 
Moreover, since $\Lambda_1: H_0^{- 1/2}(\partial D) \to H_0^{1/2}(\partial D)$ is self-adjoint, it follows that
\[ \begin{split}
 (u - U) (x)&=  \int_{\partial D}  g(y)  \; \Gamma_R(x,y)  \text{ d} s_y -  \int_{\partial D} \Lambda_1^{-1} [\Gamma_R ( x ,\cdot)](y)  \; \Lambda_{\gamma} [g] (y)    \text{ d} s_y \\
 & = \int_{\partial D}  g(y)  \; \Gamma_R(x,y)\text{ d} s_y -  \int_{\partial D} \Gamma_R ( x , y )  \; \Lambda_1^{-1}  \Lambda_{\gamma} [g] (y)  \text{ d} s_y\\
 & = \int_{\partial D}  \Gamma_R(x,y)  (g(y) -  \Lambda_1^{-1}  \Lambda_{\gamma} [g] (y) )\text{ d} s_y
 \\ & = \int_{\partial D}  \Gamma_R(x,y)  (I -  \Lambda_1^{-1}  \Lambda_{\gamma} )[g](y) \text{ d} s_y \\ & = \int_{\partial D}  \Gamma_R(x,y)  \Lambda_1^{-1}  (\Lambda_1 -   \Lambda_{\gamma} )[g](y) \text{ d} s_y \,.
\end{split} \]
Therefore, the following result holds.
\begin{lem} \label{lem:bir_1} For $x \in \RR^2 \setminus (\overline{\Omega} \cup \overline{D} )$, we have
	\begin{equation} \label{eq:bdy_int_rep} (u - U) (x) = \int_{\partial D}  \Gamma_R(x,y)  \Lambda_1^{-1}  (\Lambda_1 -   \Lambda_{\gamma} )[g](y) \text{ d} s_y,\end{equation}
	with $g= {\partial u}/{\partial \nu} |_{+}$. 
\end{lem}

\begin{thm}[Dipolar approximation] If $D = z + \delta B$, with $\text{dist}(\partial \Omega , z) \gg 1$, $\delta \ll 1$ and $B$ is a bounded open set, then for any $x \in \partial \Omega$, 
\begin{equation}\label{thm:dip_exp} \left (  \frac{\partial u}{\partial \nu} - \frac{\partial U}{\partial \nu} \right ) (x) = - \delta^2 \nabla U(z)^T M (\wh{\gamma}, B) \nabla_y \left ( \frac{\partial \Gamma_R}{\partial \nu}  \biggr |_+ \right ) (x, z)  + o(\delta^2), \end{equation}
where $T$ denotes the transpose, $M(\wh{\gamma}, B) = (m_{ij})_{i,j \in\{1,2\}}$ is the first-order polarization tensor associated with $B$ and $\wh{\gamma}$, given by
\begin{equation}\label{eq:PT-fo} m_{ij}  = \int_{\partial B}  y_{i} \mathcal{T}_B  \left ( I - \left ( \frac{I}{2} + \mathcal{K}_B^* \right ) \mathcal{T}_B  \right )^{-1} \left ( \frac{\partial x_j}{\partial \nu} \biggr |_{\partial B} \right ) (y) \text{ d} s_y , \end{equation}
$I$ is the identity operator, and $\mathcal{T}_B: H^{-1/2}_0(\partial B) \arr H^{-1/2}_0(\partial B)$ is the operator defined by $\mathcal{T}_B:= \Lambda_1^{-1}  (\Lambda_1 -   \Lambda_{\wh{\gamma}} )$.

\end{thm}

\begin{proof} Let $\Gamma$ be the fundamental solution of the Laplacian in $\RR^2$. Following \cite{pt1,Am4}, define
\[ H = - \int_{\partial \Omega }  \left ( \frac{\partial \Gamma}{\partial \nu_y} \; u -   \frac{\partial u}{\partial \nu}\; \Gamma \right )   \text{d} s  = - \mathcal{D}_{\Omega} [ u |_+] + \mathcal{S}_{\Omega} \left [ \frac{\partial u}{\partial \nu} \biggr |_+\right ].\]
Integration by parts and using the same arguments as those in the proof of Lemma \ref{lem:bir_1} yields
	\begin{equation} \label{eq:bdy_int_rep_H} (u - H) (x) = \int_{\partial D}  \Gamma(x,y)  \Lambda_1^{-1}  (\Lambda_1 -   \Lambda_{\gamma} )[g](y) \text{ d} s_y   = \mathcal{S}_{D} \left [ \mathcal{T}_D [g] \right ] (x) , \end{equation}
where $\mathcal{T}_D:= \Lambda_1^{-1}  (\Lambda_1 -   \Lambda_{\gamma} )$.

\medskip

Taking the normal derivative on $\partial D$  in \eqref{eq:bdy_int_rep_H} from outside and using the jump relations gives
\[ \label{eq:bdy_eqn_for_g} g =  \frac{\partial u}{\partial \nu} \biggr |_+ = \frac{\partial H}{\partial \nu} + \left ( \frac{I}{2} + \mathcal{K}_D^* \right )  \mathcal{T}_D[g] . \]
Hence, 
\begin{equation} \label{eq:bdy_eqn_g_obtained} g =  \left ( I - \left ( \frac{I}{2} + \mathcal{K}_D^* \right )  \mathcal{T}_D \right )^{-1} \left ( \frac{\partial H}{\partial \nu} \biggr |_{\partial D} \right ) . \end{equation}
Substituting \eqref{eq:bdy_eqn_g_obtained} into \eqref{eq:bdy_int_rep}, we get
	\begin{equation*} \label{eq:bdy_int_rep_final} (u - U) (x) = \int_{\partial D}  \Gamma_R(x,y)  \mathcal{T}_D \left ( I - \left ( \frac{I}{2} + \mathcal{K}_D^* \right )  \mathcal{T}_D \right )^{-1} \left ( \frac{\partial H}{\partial \nu} \biggr |_{\partial D} \right ) (y) \text{ d} s_y .\end{equation*}
Following the same arguments as those in \cite{pt1,Am4}, we can establish by using the scaling properties of  $\mathcal{K}_D^* $ and $\mathcal{T}_D$ that $\|\nabla H - \nabla U \| = O(\delta^2)$, see \cite{Am4}. Therefore, Taylor expanding $\Gamma_R(x,\, \cdot \,)$ and $H$ at $z$ and scaling the integral, we get the desired expression for the leading-order term  of the small-volume expansion.
\end{proof}

\begin{defn} Let $\alpha , \beta \in \NN^2$ be multi-indices. We define the generalized polarization tensors associated to the conductivity distribution $\wh{\gamma}$ by
\begin{equation}\label{eq:GPT} M_{\alpha \beta}(\wh{\gamma}, B)  =  \int_{\partial B}  y^{\alpha}  \mathcal{T}_B    \left ( I - \left ( \frac{I}{2} + \mathcal{K}_B^* \right )\mathcal{T}_B     \right )^{-1} \left ( \frac{\partial x^\beta}{\partial \nu} \biggr |_{\partial B} \right ) (y) \text{ d} s_y . \end{equation}
\end{defn}

We can also define the contracted generalized polarization tensors (CGPTs) as follows.

\begin{defn} \label{def:cgpts} Let $z = y_1+ i y_2$ and $\zeta = x_1 + i x_2$.  For any pair of indices $m, n \in \NN$,  we define
\[ M^{cc}_{mn} (\wh{\gamma}, B):=  \int_{\partial B}  \text{Re}(z^m)  \mathcal{T}_B  \left ( I - \left ( \frac{I}{2} + \mathcal{K}_B^* \right )  \mathcal{T}_B   \right )^{-1} [ \text{Re}(\zeta^n) ] (y) \text{ d} s_y , \]
\[ M^{cs}_{mn}(\wh{\gamma}, B):= \int_{\partial B}  \text{Re}(z^m)  \mathcal{T}_B   \left ( I - \left ( \frac{I}{2} + \mathcal{K}_B^* \right )  \mathcal{T}_B   \right )^{-1} [ \text{Im}(\zeta^n) ] (y) \text{ d} s_y ,\]
\[ M^{sc}_{mn} (\wh{\gamma}, B):= \int_{\partial B}  \text{Im}(z^m)  \mathcal{T}_B   \left ( I - \left ( \frac{I}{2} + \mathcal{K}_B^* \right )  \mathcal{T}_B    \right )^{-1} [ \text{Re}(\zeta^n) ](y)  \text{ d} s_y ,\]
\[ M^{ss}_{mn} (\wh{\gamma}, B):=  \int_{\partial B}  \text{Im}(z^m)  \mathcal{T}_B   \left ( I - \left ( \frac{I}{2} + \mathcal{K}_B^* \right )  \mathcal{T}_B   \right )^{-1} [ \text{Im}(\zeta^n) ] (y) \text{ d} s_y .  \]
\end{defn}

Note that the CGPTs introduced here coincide with those studied in \cite{Am1}. 

\begin{prop}\label{prop:equiv_ident} Let $u$ be the solution to 
\begin{equation} \label{eq:problem_Am1} \begin{cases} \nabla \cdot \gamma \nabla u = 0  & \mbox{in } \RR^2 , \\ u - h = O(|x|^{-1})   & \mbox{as } |x| \to + \infty , \end{cases} \end{equation}
where $h$ is a harmonic function in $\RR^2$. Then the following identity holds:
\begin{equation}   \label{eq:identity_equivalence}\gamma \frac{\partial u}{\partial \nu}  \biggr |_- =   \left ( I - \left ( \frac{I}{2} + \mathcal{K}_D^* \right )  \mathcal{T}_D   \right )^{-1} \left [\frac{\partial  h}{\partial \nu} \right ] \,  \qquad \mbox{ on } \partial D.\end{equation}
\end{prop}
\begin{proof} The solution $u$ to \eqref{eq:problem_Am1} can be represented as
\begin{equation} \label{eq:ext_representation_prop} u = h + \mathcal{S}_{D}[\psi]   \qquad\text{ in } \RR^2 \setminus \overline{D},\end{equation}
for some $\psi \in L^2(\partial D)$. Therefore,
\[ \frac{\partial u}{\partial \nu} \biggr |_+ = \frac{\partial h}{\partial \nu}  + \left ( \frac{I}{2} + \mathcal{K}_{D}^* \right )[\psi] .\]
The transmission condition on $\partial D$ 
\[ \frac{\partial u}{\partial \nu} \biggr |_+ = \gamma  \frac{\partial u}{\partial \nu} \biggr |_-  \]
leads to
\begin{equation} \label{eq:equation_on_bdy_prop} \gamma \frac{\partial u}{\partial \nu} \biggr |_- = \frac{\partial h}{\partial \nu}  + \left ( \frac{I}{2} + \mathcal{K}_{D}^* \right )[\psi]. \end{equation}
Applying the map $\Lambda_\gamma$ on both sides gives
\[ \left ( \frac{I}{2} + \mathcal{K}_{D}^* \right )[\psi] = u  -  \Lambda_\gamma \left [ \frac{\partial h}{\partial \nu} \right ] . \]
Next, using the representation \eqref{eq:ext_representation_prop}, we get
\[ \Lambda_\gamma \left ( \frac{I}{2} + \mathcal{K}_{D}^* \right )[\psi] = h +  \mathcal{S}_{D}[\psi]  -  \Lambda_\gamma \left [ \frac{\partial h}{\partial \nu} \right ] . \]
Applying $\Lambda_1^{-1}$ and using the jump relations, we obtain
\[ \Lambda_1^{-1} \Lambda_\gamma \left ( \frac{I}{2} + \mathcal{K}_{D}^* \right )[\psi] =  \frac{\partial h}{\partial \nu} +  \left ( - \frac{I}{2} + \mathcal{K}_{D}^* \right ) [\psi]  -  \Lambda_1^{-1} \Lambda_\gamma \left [ \frac{\partial h}{\partial \nu} \right ] . \]

\[ \left [ - \left (  - \frac{I}{2}  + \mathcal{K}_{D}^* \right )  +  \Lambda_1^{-1} \Lambda_\gamma \left ( \frac{I}{2} + \mathcal{K}_{D}^* \right ) \right ] [\psi] =  \frac{\partial h}{\partial \nu}   -  \Lambda_1^{-1} \Lambda_\gamma \left [ \frac{\partial h}{\partial \nu} \right ], \]
\[ \left [ I - \left (  \frac{I}{2}  + \mathcal{K}_{D}^* \right )  +  \Lambda_1^{-1} \Lambda_\gamma \left ( \frac{I}{2} + \mathcal{K}_{D}^* \right ) \right ] [\psi] =  \mathcal{T}_D \left [ \frac{\partial h}{\partial \nu} \right ] ,\]
\[ \left [ I - \mathcal{T}_D \left (  \frac{I}{2}  + \mathcal{K}_{D}^* \right ) \right ] [\psi] =  \mathcal{T}_D \left [ \frac{\partial h}{\partial \nu} \right ] . \]
Therefore, we get the following expression of $\psi$:  
\[ \psi = \left ( I - \mathcal{T}_D \left (  \frac{I}{2}  + \mathcal{K}_{D}^* \right ) \right )^{-1} \left [ \mathcal{T}_D \left [ \frac{\partial h}{\partial \nu} \right ]  \right ]. \]
Substituting $\psi$ into \eqref{eq:equation_on_bdy_prop}, we arrive at 
\[\gamma \frac{\partial u}{\partial \nu} \biggr |_- = \frac{\partial h}{\partial \nu}  + \left ( \frac{I}{2} + \mathcal{K}_{D}^* \right ) \left ( I - \mathcal{T}_D \left (  \frac{I}{2}  + \mathcal{K}_{D}^* \right ) \right )^{-1}  \mathcal{T}_D \left [ \frac{\partial h}{\partial \nu} \right ]  ,\]
or equivalently, 
\begin{equation} \label{eq:rhs_g_prop} \gamma \frac{\partial u}{\partial \nu} \biggr |_- = \left [ I  - \left (   I - \left ( \left ( \frac{I}{2} + \mathcal{K}_{D}^* \right ) \mathcal{T}_D \right )^{-1}  \right )^{-1} \right ] \left [ \frac{\partial h}{\partial \nu} \right ],  \end{equation}
which  is equivalent to \eqref{eq:identity_equivalence}.
\end{proof}

The following result shows that the GPTs are the building blocks of the multipolar asymptotic expansion of $u-U$.

\begin{thm}[Multipolar approximation] For any $x \in \partial \Omega$, 
\begin{equation} \label{thm:dip_expm} 
\left (  \frac{\partial u}{\partial \nu} - \frac{\partial U}{\partial \nu} \right ) (x) = - 
\sum_{|\alpha|, |\beta| \geq 1}  \frac{1}{\alpha! \beta!} \delta^{|\alpha|+|\beta|} \partial^\alpha U(z) M_{\alpha\beta} (\wh{\gamma}, B) \partial^\beta_y 
\left ( \frac{\partial \Gamma_R}{\partial \nu}  \biggr |_+ \right ) (x, z). 
\end{equation}
\end{thm}

\begin{rem}
In view of \eqref{thm:dip_expm}, the GPTs $M_{\alpha\beta} (\wh{\gamma}, B)$ can be reconstructed from measurements of ${\partial u}/{\partial \nu} - {\partial U}/{\partial \nu}$ corresponding to different positions of the fish. As in \cite{Am3}, the number of GPTs which can be reconstructed accurately for a given signal-to-noise ratio can be determined in terms of the ratio between the characteristic size  of the target and its distance to the fish. The resolving formula derived in \cite{Am3} holds. 
Moreover, it is worth mentioning that if the target is  homogeneous, then the GPTs reduce to those first introduced and investigated in \cite{pt1,pt2}. In fact, if $\gamma |_{D} \equiv k$, $ 0< k \ne 1 < +\infty$, then
\[\begin{split} M_{\alpha \beta}(k, D)  &=  \int_{\partial D}  y^{\alpha}   \left ( \mathcal{T}^{-1}_{D} - \left ( \frac{I}{2} + \mathcal{K}_D^* \right )   \right )^{-1} \left ( \frac{\partial x^\beta}{\partial \nu} \biggr |_{\partial D} \right ) (y) \text{ d} s_y\\&  =  \int_{\partial D}  y^{\alpha}   \left ( \lambda I - \mathcal{K}_D^*  \right )^{-1} \left ( \frac{\partial x^\beta}{\partial \nu} \biggr |_{\partial D} \right ) (y) \text{ d} s_y,\end{split} \]
where $\lambda = \frac{k + 1}{2(k-1)}$.

\end{rem}

\section{Properties of the CGPTs }  \label{sec2}
The goal of this section is to provide transformation formulas for the contracted GPTs introduced in Definition \ref{def:cgpts}.

\subsection{Translation formula}

We want to investigate how the quantity 
\begin{equation} \label{eq:original} M^{cc}_{mn} (\gamma , D) = M_{mn} =  \int_{\partial D}  \text{Re}(z^m)  \mathcal{T}_D  [ g_n^c](y) \text{ d} s_y  \,\end{equation}
changes with respect to a translation of $D$. 

\medskip

We denote by $\wh{x} = x + z$, $\wh{D}:= D + z$, $\wh{1}(\wh{x}) := 1$,  and $\wh{\gamma}(\wh{x}):= \gamma (x)$. We want to relate $M^{cc}_{mn} (\gamma , D)$ with $ M^{cc}_{mn} (\wh{\gamma} , \wh{D})$ defined by
\begin{equation*}\label{eq:transl} M^{cc}_{mn} (\wh{\gamma}, \wh{D}) = \wh{M}_{mn}^{cc} =  \int_{\partial \wh{D}}  \text{Re}(\wh{x}^m)  (  \wh{g}_n^c(\wh{x}) -   \Lambda_{\wh{1}}^{-1}\Lambda_{\wh{\gamma}}^{\wh{D}} [ \wh{g}_n^c](\wh{x}) ) \text{ d} s_{\wh{x}} . \end{equation*}
By the change of variables $\wh{x} = x + z$, we obtain
\begin{equation} \label{eq:transl_2} M^{cc}_{mn} (\wh{\gamma} , \wh{D}) = \wh{M}_{mn}^{cc} =  \int_{\partial D}  \text{Re}((x+z)^m)  (  \wh{g}_n^c(x+z) -   \Lambda_{\wh{1}}^{-1}\Lambda_{\wh{\gamma}}^{\wh{D}} [ \wh{g}_n^c](x+z) ) \text{ d} s_{x} \,. \end{equation}
\begin{lem}\label{lemma:gnc_translation} We have
\begin{equation} \label{eq:formula_gnc_translation} \wh{g}_n^c(x+z) = \sum_{k=1}^n \left ( \begin{matrix} n \\ k \end{matrix} \right )  [ g^c_k(x)   r_z^{n-k} \cos ( (n-k ) \theta_z) - g^s_k(x)  r_z^{n-k}  \sin ( (n-k ) \theta_z) ] \,, 
\end{equation}
and
\begin{equation} \label{eq:formula_gns_translation} \wh{g}_n^s(x+z) = \sum_{k=1}^n \left ( \begin{matrix} n \\ k \end{matrix} \right )  [ g^s_k(x)   r_z^{n-k} \cos ( (n-k ) \theta_z) + g^c_k(x)  r_z^{n-k}  \sin ( (n-k ) \theta_z) ] \,,
\end{equation}
where $z= r_z (\cos \theta_z, \sin \theta_z)$ in polar coordinates.  
\end{lem}

\begin{proof} Let $\wh{u}^c_n$ be the solution to
\begin{equation} \label{eq:lemma_gnc_translation} 
\begin{cases} \nabla_{\wh{x}} \cdot \wh{\gamma}(\wh{x}) \nabla_{\wh{x}} \wh{u}^c_n (\wh{x}) = 0 , &\qquad \mbox{in } \RR^2  \\  \wh{u}^c_n (\wh{x}) - \text{Re}(\wh{x}^n)   = O(|{\wh{x}}|^{-1})  &  \qquad \mbox{as } |\wh{x}| \to + \infty ,\end{cases} \end{equation}
Then, by definition, $\wh{g}_n^c:= \wh{\gamma} \dfrac{\partial \wh{u}^c_n}{\partial \nu_{\wh{x}}}$. Using the change of variables in \eqref{eq:lemma_gnc_translation} and setting  $v^c_n(x):= \wh{u}^c_n  (x+z)$,  we obtain
\begin{equation*} \label{eq:lemma_gnc_translation_2} \begin{cases} \nabla_{{x}} \cdot {\gamma}({x}) \nabla_{{x}} v^c_n  (x) = 0 &\qquad \mbox{in } \RR^2 , \\  {v}^c_n ({x}) - \text{Re}({(x+z)}^n) = O(|x|^{-1})  & \qquad \mbox{as } |x| \to + \infty .\end{cases} \end{equation*}
From
\[ \text{Re} ( ( x + z )^n ) = \sum_{k=0}^n \left ( \begin{matrix} n \\ k \end{matrix} \right )  r_x^k r_z^{n-k} [ \cos (k \theta_x) \cos ( (n-k ) \theta_z) -  \sin (k \theta_x) \sin ( (n-k ) \theta_z) ] , \]
it follows that
\[ {v}^c_n ({x}) = \sum_{k=0}^n \left ( \begin{matrix} n \\ k \end{matrix} \right )  [ h^c_k(x)   r_z^{n-k} \cos ( (n-k ) \theta_z) - h^s_k(x)  r_z^{n-k}  \sin ( (n-k ) \theta_z) ] . \]
Hence, 
\[ \wh{g}^c_n(x+z) = \sum_{k=1}^n \left ( \begin{matrix} n \\ k \end{matrix} \right )  [ g^c_k(x)   r_z^{n-k} \cos ( (n-k ) \theta_z) - g^s_k(x)  r_z^{n-k}  \sin ( (n-k ) \theta_z) ] .\]
Analogously we derive formula \eqref{eq:formula_gns_translation} for $\wh{g}_n^s$.
\end{proof}
 
 To relate \eqref{eq:original} and \eqref{eq:transl_2} we consider the operator \[ \Lambda_1^{-1}\Lambda_{{\gamma}}^{{D}}: H^{-1/2}_0(\partial {D}) \arr H^{-1/2}_0(\partial {D})\] 
\[ {g}_n^c \longmapsto \frac{\partial  {h}}{\partial \nu_{{x}} } \biggr |_{-} , \]
 where $ {h}$ is the solution to the boundary value problem
\begin{equation} \label{eq:bvp_transl} \begin{cases} \Laplacian {h} = 0  &\quad \mbox{in } {D} , \\ \nabla_x h \cdot \nu_x =  {g}_n^c   &\quad \mbox{on } \partial {D} .\end{cases} \end{equation}

Let ${y} \in {D}$ and consider the corresponding Neumann function ${N}_1({x},{y})$, that is, the solution to
\begin{equation} \label{eq:N_bvp_transl} \begin{cases} \Laplacian_{{x}} {N}_1 ({x},{y}) = - \delta_y(x) , &  \qquad {x}\in {D} , \\ \nabla_{{x}} {N}_1({x},{y}) \cdot \nu_{{x}} = \frac{1}{|\partial {D}|} ,  & \qquad {x} \in \partial {D} , \\ \displaystyle \int_{\partial D} N_1(x,y) \dds_x = 0 .\end{cases} \end{equation}

Then the solution $ {h}$ of \eqref{eq:bvp_transl}  can be represented by means of this Neumann function
\[  {h}({y}) = \int_{\partial {D}} {N}_1 ({x} , {y} )  {g}({x}) \text{ d} s_{{x}} . \]
So, for $y \in \partial D$, we have
\begin{equation} \label{eq:integral_representation_original_setting} \Lambda_1^{-1}\Lambda_{{\gamma}}^{{D}} [g_n^c] (y) = \nabla_y h(y) \cdot \nu_y |_{\partial D} =  \int_{\partial {D}} \nabla_{y} {N}_1 ({x} , {y} ) \cdot \nu_{y}  \,{g}_n^c({x}) \text{ d} s_{{x}} .\end{equation}

\medskip

Now we proceed similarly. We let the operator $\Lambda_{\wh{1}}^{-1}\Lambda_{\wh{\gamma}}^{\wh{D}}: H^{-1/2}_0(\partial \wh{D}) \arr H^{-1/2}_0(\partial \wh{D})$ be defined by
\[ \wh{g}_n^c \longmapsto \frac{\partial  \wh{h}}{\partial \nu_{\wh{x}} } \biggr |_{-}  , \]
with $ \wh{h}$ being the solution of
\begin{equation} \label{eq:bvphat_transl} \begin{cases} \Laplacian \wh{h} = 0 & \quad \mbox{in } \wh{D} ,\\ \nabla_{\wh{x}} \wh{h} \cdot {\nu_{\wh{x}}} =  \wh{g}_n^c  & \quad \mbox{on } \partial \wh{D}.\end{cases} \end{equation}

Let $\wh{y} = y + z \in \wh{D}$ and consider the corresponding Neumann function $\wh{N}_1(\wh{x},\wh{y})$, that is,  the solution to
\begin{equation*} \label{eq:N_bvphat_transl} \begin{cases} \Laplacian_{\wh{x}} \wh{N}_1 (\wh{x},\wh{y}) = - \delta_{\wh{y}}(\wh{x}),  &  \qquad \wh{x}\in \wh{D}, \\ \nabla_{\wh{x}} \wh{N}_1(\wh{x},\wh{y}) \cdot \nu_{\wh{x}} =- \frac{1}{|\partial \wh{D}|} , & \qquad \wh{x} \in \partial \wh{D}, \\\displaystyle \int_{\partial \wh{D}} \wh{N}_1(\wh{x},\wh{y}) \dds_{\wh{x}} = 0, \end{cases} \end{equation*}
that can be written as
\begin{equation} \label{eq:N_bvphat_transl_2} \begin{cases} \Laplacian_{{x}} \wh{N}_1 (x + z,y + z) = - \delta_{y}(x) , &  \qquad {x}\in {D} ,\\ \nabla_{{x}} \wh{N}_1({x} + z,{y}+z) \cdot \nu_{{x}} = -\frac{1}{|\partial {D}|} , & \qquad {x} \in \partial {D} ,\\ \displaystyle \int_{\partial {D}} \wh{N}_1({x}+z,{y}+z) \dds_{{x}} = 0. \end{cases} \end{equation}
Comparing \eqref{eq:N_bvphat_transl_2} and \eqref{eq:N_bvp_transl},  we observe that  $\wh{N}_1 (x + z,y + z)$ and ${N}_1 (x,y)$ satisfy the same boundary value problem \eqref{eq:N_bvp_transl}. The uniqueness of a solution to \eqref{eq:N_bvp_transl} yields
\begin{equation} \label{eq:identity_Nfc_transl} \wh{N}_1 (x + z,y + z) = {N}_1 (x,y) \, .\end{equation}

The solution $ \wh{h}$ to \eqref{eq:bvphat_transl}  can be represented by means of the Neumann function $\wh{N}_1$:
\[ \wh{h}(\wh{y}) = \int_{\partial \wh{D}} \wh{N}_1 (\wh{x} , \wh{y} )  \,\wh{g}_n^c(\wh{x}) \text{ d} s_{\wh{x}} .  \]
Moreover, for $y \in \partial D$, we have
\[ \begin{split} \Lambda_{\wh{1}}^{-1} \Lambda_{\wh{\gamma}}^{\wh{D}} [\wh{g}_n^c] (\wh{y} )  & = \nabla_{\wh{y}}\,\wh{h}(\wh{y}) \cdot \nu_{\wh{y}} \\ &= \int_{\partial \wh{D}}  (\nabla_{\wh{y}} \, \wh{N}_1) (\wh{x} , \wh{y} ) \cdot \nu_{\wh{y}} \,\wh{g}_n^c(\wh{x}) \text{ d} s_{\wh{x}} \\ & = \int_{\partial {D}}  (\nabla_{y} \, \wh{N}_1) ({x} + z , \wh{y}) \cdot \nu_{{y}} \, \wh{g}_n^c(x + z) \text{ d} s_{{x}} \\ & = \int_{\partial {D}}  (\nabla_{y} \, {N}_1) ({x} , y) \cdot \nu_{{y}} \, \sum_{k=1}^n \left ( \begin{matrix} n \\ k \end{matrix} \right )  [ g^c_k(x)   r_z^{n-k} \cos ( (n-k ) \theta_z) - g^s_k(x)  r_z^{n-k}  \sin ( (n-k ) \theta_z) ]\text{ d} s_{{x}} 
\\ & =  \sum_{k=1}^n \left ( \begin{matrix} n \\ k \end{matrix} \right )  r_z^{n-k} \left [ \cos ( (n-k ) \theta_z) \int_{\partial {D}}  (\nabla_{y} \, {N}_1) ({x} , y) \cdot \nu_{{y}} \,  g^c_k(x)  \text{ d} s_{{x}} \right . \\ & \left . -   \sin ( (n-k ) \theta_z) \,\int_{\partial {D}}  (\nabla_{y} \, {N}_1) ({x} , y) \cdot \nu_{{y}} \,  g^s_k(x) \text{ d} s_{{x}} \right ] 
\\ & =  \sum_{k=1}^n \left ( \begin{matrix} n \\ k \end{matrix} \right )  r_z^{n-k} \left [ \cos ( (n-k ) \theta_z)\,  \Lambda_1^{-1}\Lambda_{{\gamma}}^{{D}} [g^c_k] (y) -   \sin ( (n-k ) \theta_z) \,\Lambda_1^{-1}\Lambda_{{\gamma}}^{{D}} [g^s_k] (y) \right ] \,.\end{split} \]
Here we have made the change of variables $\wh{x} = x + z$ and  used  identity \eqref{eq:identity_Nfc_transl}. Therefore, it follows that
\begin{equation} \label{eq:lambda_translation} \Lambda_{\wh{1}}^{-1} \Lambda_{\wh{\gamma}}^{\wh{D}} [\wh{g}] (\wh{y} )  = \sum_{k=1}^n \left ( \begin{matrix} n \\ k \end{matrix} \right )  r_z^{n-k} \left [ \cos ( (n-k ) \theta_z)\,  \Lambda_1^{-1}\Lambda_{{\gamma}}^{{D}} [g^c_k] (y) -   \sin ( (n-k ) \theta_z) \,\Lambda_1^{-1}\Lambda_{{\gamma}}^{{D}} [g^s_k] (y) \right ] . \end{equation}

Hence\[  M^{cc}_{mn} (\wh{\gamma} , \wh{D})  =  \int_{\partial D}  \Re ((x+z)^m)  (  g_n^c(x) -   \Lambda_1^{-1}\Lambda_{\gamma}^{D} [ g_n^c](x) ) \text{ d} s_{x} .\]
Using the identity 
\[ \Re( ( x + z )^m ) = \sum_{k=0}^m \left ( \begin{matrix} m \\ k \end{matrix} \right )  r_x^k r_z^{m-k} [ \cos (k \theta_x) \cos ( (m-k ) \theta_z) -  \sin (k \theta_x) \sin ( (m-k ) \theta_z) ] \,,\]
we get
\begin{equation} \label{eq:gpt_translation_a} \begin{split} M^{cc}_{mn} (\wh{\gamma} , \wh{D})  = \int_{\partial D}  \Re ((x+z)^m)  (  \wh{g}_n^c(x+z) -   \Lambda_{\wh{1}}^{-1}\Lambda_{\wh{\gamma}}^{\wh{D}} [ \wh{g}_n^c](x+z) ) \text{ d} s_{x} 
\\ = \sum_{k=1}^m r_z^{m-k} \left ( \begin{matrix} m \\ k \end{matrix} \right ) \left [  \cos ( (m-k ) \theta_z)  \int_{\partial D}  r_x^k  \cos (k \theta_x) (  \wh{g}_n^c(x+z) -   \Lambda_{\wh{1}}^{-1}\Lambda_{\wh{\gamma}}^{\wh{D}} [ \wh{g}_n^c ](x+z) ) \text{ d} s_{x} \right . \\  \left . - \sin ( (m-k ) \theta_z)  \int_{\partial D}  r_x^k  \sin (k \theta_x) (  \wh{g}_n^c(x+z) -   \Lambda_{\wh{1}}^{-1}\Lambda_{\wh{\gamma}}^{\wh{D}} [ \wh{g}_n^c](x+z) ) \text{ d} s_{x} \right ] .
  \end{split} \end{equation}
From Lemma \ref{lemma:gnc_translation} and formula \eqref{eq:lambda_translation}, we obtain 
\[\begin{split}  \wh{g}_n^c(x+z) -   \Lambda_{\wh{1}}^{-1}\Lambda_{\wh{\gamma}}^{\wh{D}} [ \wh{g}_n^c ](x+z)  = \sum_{r=1}^n \left ( \begin{matrix} n \\ r\end{matrix} \right )  r_z^{n-r} [ g^c_r(x)   \cos ( (n-r ) \theta_z) - g^s_r(x)  \sin ( (n-r ) \theta_z) ] \\ - \sum_{r=1}^n \left ( \begin{matrix} n \\ r \end{matrix} \right )  r_z^{n-r} \left [ \cos ( (n-r ) \theta_z)\,  \Lambda_1^{-1}\Lambda_{{\gamma}}^{{D}} [g^c_r] (x) -   \sin ( (n-r ) \theta_z) \,\Lambda_1^{-1}\Lambda_{{\gamma}}^{{D}} [g^s_r] (x) \right ]
\\ = \sum_{r=1}^n \left ( \begin{matrix} n \\ r \end{matrix} \right )  r_z^{n-r} \left [ g^c_r(x)   \cos ( (n-r ) \theta_z) - g^s_r(x)  \sin ( (n-r ) \theta_z)  \right . \\  \left . - ( \cos ( (n-r ) \theta_z)\,  \Lambda_1^{-1}\Lambda_{{\gamma}}^{{D}} [g^c_r] (x) -   \sin ( (n-r ) \theta_z) \,\Lambda_1^{-1}\Lambda_{{\gamma}}^{{D}} [g^s_r] (x) ) \right ]   \\
=  \sum_{r=1}^n \left ( \begin{matrix} n \\ r \end{matrix} \right )  r_z^{n-r} \left [ (g^c_r(x) - \Lambda_1^{-1}\Lambda_{{\gamma}}^{{D}} [g^c_r] (x) )   \cos ( (n-r ) \theta_z) 
 - ( g^s_r(x) - \Lambda_1^{-1}\Lambda_{{\gamma}}^{{D}} [g^s_r] (x) ) \sin ( (n-r ) \theta_z)  \right ] . \end{split}  
\]
So
\begin{equation}\label{eq:integral_translation_1} \begin{split} \int_{\partial D}  r_x^k  \cos (k \theta_x) (  \wh{g}_n^c(x+z) -   \Lambda_{\wh{1}}^{-1}\Lambda_{\wh{\gamma}}^{\wh{D}} [ \wh{g}_n^c ](x+z) ) \text{ d} s_{x} \\
 =
   \sum_{r=1}^n \left ( \begin{matrix} n \\ r \end{matrix} \right )  r_z^{n-r} \left [  \cos ( (n-r ) \theta_z)M^{cc}_{kr}  - \sin ( (n-r ) \theta_z) M^{cs}_{kr}  \right ] , \end{split} \end{equation}
   and
\begin{equation} \label{eq:integral_translation_2} \begin{split} \int_{\partial D}  r_x^k  \sin (k \theta_x) (  \wh{g}_n^c(x+z) -   \Lambda_{\wh{1}}^{-1}\Lambda_{\wh{\gamma}}^{\wh{D}} [ \wh{g}_n^c](x+z) ) \text{ d} s_{x}  \\
 = \sum_{r=1}^n \left ( \begin{matrix} n \\ r \end{matrix} \right )  r_z^{n-r} \left [  \cos ( (n-r ) \theta_z)M^{sc}_{kr}  - \sin ( (n-r ) \theta_z) M^{ss}_{kr}  \right ] . \end{split} \end{equation}

Plugging \eqref{eq:integral_translation_1} and \eqref{eq:integral_translation_2} into formula \eqref{eq:gpt_translation_a}, we arrive at
\begin{equation*} \label{eq:gpt_translation_cc} \begin{split} M^{cc}_{mn} (\wh{\gamma} , \wh{D})  &= \sum_{k=1}^m    \sum_{r=1}^n \left ( \begin{matrix} n \\ r \end{matrix} \right )  \left ( \begin{matrix} m \\ k \end{matrix} \right ) r_z^{n-r} r_z^{m-k}  \left \{  \cos ( (m-k ) \theta_z) \left [  \cos ( (n-r ) \theta_z)M^{cc}_{kr}  - \sin ( (n-r ) \theta_z) M^{cs}_{kr}  \right ] \right . \\ & \left . - \sin ( (m-k ) \theta_z)  \ \left [  \cos ( (n-r ) \theta_z)M^{sc}_{kr}  - \sin ( (n-r ) \theta_z) M^{ss}_{kr}  \right ] \right \} .
  \end{split} \end{equation*}

Analogously, we readily get
\begin{equation*} \label{eq:gpt_translation_sc} \begin{split} M^{sc}_{mn} (\wh{\gamma} , \wh{D})  &= \sum_{k=1}^m    \sum_{r=1}^n \left ( \begin{matrix} n \\ r \end{matrix} \right )  \left ( \begin{matrix} m \\ k \end{matrix} \right ) r_z^{n-r} r_z^{m-k}  \left \{  \cos ( (m-k ) \theta_z) \left [  \cos ( (n-r ) \theta_z)M^{sc}_{kr}  - \sin ( (n-r ) \theta_z) M^{ss}_{kr}  \right ] \right . \\ & \left . + \sin ( (m-k ) \theta_z)  \left [  \cos ( (n-r ) \theta_z)M^{cc}_{kr}  - \sin ( (n-r ) \theta_z) M^{cs}_{kr}  \right ] \right \} ,
  \end{split} \end{equation*}

\begin{equation*} \label{eq:gpt_translation_cs} \begin{split} M^{cs}_{mn} (\wh{\gamma} , \wh{D}) &= \sum_{k=1}^m    \sum_{r=1}^n \left ( \begin{matrix} n \\ r \end{matrix} \right )  \left ( \begin{matrix} m \\ k \end{matrix} \right ) r_z^{n-r} r_z^{m-k}  \left \{  \cos ( (m-k ) \theta_z) \left [  \cos ( (n-r ) \theta_z)M^{cs}_{kr}  + \sin ( (n-r ) \theta_z) M^{cc}_{kr}  \right ] \right . \\ & \left . - \sin ( (m-k ) \theta_z)  \ \left [  \cos ( (n-r ) \theta_z)M^{ss}_{kr}  + \sin ( (n-r ) \theta_z) M^{sc}_{kr}  \right ] \right \} ,
  \end{split}  \end{equation*}

\begin{equation*} \label{eq:gpt_translation_ss} \begin{split} M^{ss}_{mn} (\wh{\gamma} , \wh{D})  &= \sum_{k=1}^m    \sum_{r=1}^n \left ( \begin{matrix} n \\ r \end{matrix} \right )  \left ( \begin{matrix} m \\ k \end{matrix} \right ) r_z^{n-r} r_z^{m-k}  \left \{  \cos ( (m-k ) \theta_z) \left [  \cos ( (n-r ) \theta_z)M^{ss}_{kr}  + \sin ( (n-r ) \theta_z) M^{sc}_{kr}  \right ] \right . \\ & \left . + \sin ( (m-k ) \theta_z)  \ \left [  \cos ( (n-r ) \theta_z)M^{cs}_{kr}  + \sin ( (n-r ) \theta_z) M^{cc}_{kr}  \right ] \right \}.
  \end{split} \end{equation*}

We write these formulas compactly in a matrix form as follows:
\[  { \begin{bmatrix} \wh{M}^{cc}_{mn}&  \wh{M}^{sc}_{mn} \\  \wh{M}^{cs}_{mn}  &  \wh{M}^{ss}_{mn}  \end{bmatrix} = \sum_{k=1}^m \sum_{r=1}^n r_z^{m-k}  r_z^{n-r}  \left ( \begin{matrix} m \\ k \end{matrix} \right ) \left ( \begin{matrix} n\\ r \end{matrix} \right ) \mathbf{R}((n-r)\theta_z)   \begin{bmatrix} M_{k r}^{cc} & M_{k r}^{sc} \\ M_{kr}^{cs} & M_{kr}^{ss} \end{bmatrix}  \cdot \mathbf{R}((m-k)\theta_z)^T \, ,} \]
where $\mathbf{R}(\theta)$ is the matrix associated with the rotation by $\theta$, i.e.,
\begin{equation} \label{addmatrix} 
\mathbf{R}(\theta):= \begin{bmatrix} \cos \theta & - \sin \theta \\ \sin \theta & \cos \theta  \end{bmatrix} . \end{equation}

\subsection{Rotation formula}

We want to investigate how $M_{mn}$ changes with respect to a rotation of $D$ by an angle $\theta$.

\medskip

 Let $\theta \in [0, 2 \pi)$ and let $\mathbf{R} \in \RR^{2 \times 2}$ be the matrix associated with the rotation
\[ \mathbf{R}:= \mathbf{R}(\theta)  = \begin{bmatrix} \cos \theta & - \sin \theta \\ \sin \theta & \cos \theta  \end{bmatrix} . \]
Let us define $\Phi_{\mathbf{R}} (x):= \mathbf{R} x = e^{i \theta} x$.

\medskip

We denote $\wh{x} = \mathbf{R} x$, $\wh{D}:=  \Phi_{\mathbf{R}} ( D )$  and $\wh{\gamma}(\wh{x}):= \gamma (x)$. We want to relate $M^{cc}_{mn} (\gamma , D)$ with $ M^{cc}_{mn} (\wh{\gamma} , \wh{D})$ defined by 
\begin{equation}\label{eq:CGPTrotation} M^{cc}_{mn} (\wh{\gamma}, \wh{D}) = \wh{M}_{mn}^{cc} =  \int_{\partial \wh{D}}  \text{Re}({\wh{x}}^m)  (  \wh{g}_n^c(\wh{x}) -   \Lambda_{\wh{1}}^{-1}\Lambda_{\wh{\gamma}}^{\wh{D}} [ \wh{g}_n^c](\wh{x}) ) \text{ d} s_{\wh{x}}  . \end{equation}
By the change of variables $\wh{x} = \mathbf{R} x$, we obtain
\begin{equation}\label{eq:CGPTrotation_2} M^{cc}_{mn} (\wh{\gamma}, \wh{D}) = \wh{M}_{mn}^{cc}=  \int_{\partial D}  r_x^m \cos (m (\theta_x + \theta))\, (  \wh{g}_n^c(\mathbf{R} x) -   \Lambda_{\wh{1}}^{-1}\Lambda_{\wh{\gamma}}^{\wh{D}} [ \wh{g}_n^c](\mathbf{R} x) ) \text{ d} s_{x} \,. \end{equation}

\begin{lem}\label{lemma:gnc_rotation} We have
\begin{equation} \label{eq:formula_gnc_rotation}  \wh{g}_n^c(\mathbf{R} x ) =   \cos n \theta \,g^c_n(x)  - \sin n \theta  \,g^s_n(x) \,,
\end{equation}
and 
\begin{equation} \label{eq:formula_gns_rotation}  \wh{g}_n^s(\mathbf{R} x ) =   \cos n \theta \,g^s_n(x)  + \sin n \theta  \,g^c_n(x) \,. \end{equation}
\end{lem}

\begin{proof} Let $\wh{u}^c_n$ be the solution to
\begin{equation} \label{eq:lemma_gnc_rotation}  \begin{cases} \nabla_{\wh{x}} \cdot \wh{\gamma}(\wh{x}) \nabla_{\wh{x}} \wh{u}^c_n (\wh{x}) = 0  &\qquad \mbox{in } \RR^2 , \\  \wh{u}^c_n (\wh{x}) - r^{n}_{\wh{x}} \cos (n \theta_{\wh{x}})  = O(r_{\wh{x}}^{-1})  &  \qquad \mbox{as } |\wh{x}| \to + \infty . \end{cases}  \end{equation}
Then, by definition, we have $\wh{g}_n^c:=  \dfrac{\partial \wh{u}^c_n}{\partial \nu_{\wh{x}}} \biggr |_+$. By a change of variables in \eqref{eq:lemma_gnc_rotation} and by setting  $v^c_n(x):= \wh{u}^c_n  (\mathbf{R}x )$, we obtain
\[ \label{eq:lemma_gnc_rotation_2} \begin{cases} \nabla_{{x}} \cdot {\gamma}({x}) \nabla_{{x}} v^c_n  (x) = 0&\qquad \mbox{in } {\RR^2}  ,\\  {v}^c_n ({x}) - r^{n}_x \cos (n ( \theta_x + \theta)) = O(r^{-1}_x)  &  \qquad \mbox{as } |x| \to +\infty. \end{cases} \]
Hence, 
\[ \label{eq:lemma_gnc_rotation_3} \begin{cases} \nabla_{{x}} \cdot {\gamma}({x}) \nabla_{{x}} v^c_n  (x) = 0  &\qquad \mbox{in } {\RR^2}  ,\\  {v}^c_n ({x}) - (r^{n}_x \cos n \theta_x \cos n \theta - r^{n}_x \sin n \theta_x \sin n \theta ) = O(r^{-1}_x)  &  \qquad \mbox{as } |x| \to +\infty,\end{cases} \]
or equivalently, 
\[\label{eq:lemma_gnc_rotation_4} 
\begin{cases} \nabla_{{x}} \cdot {\gamma}({x}) \nabla_{{x}} v^c_n  (x) = 0 &\qquad \mbox{in } {\RR^2}  ,\\  {v}^c_n ({x}) - (h_n^c(x) \theta_x \cos n \theta - h_n^s(x) \sin n \theta )  = O(r^{-1}_x)   &  \qquad \mbox{as } |x| \to +\infty.\end{cases} \]
Therefore,
\[ {v}^c_n ({x}) =  \cos n \theta \,u^c_n(x)  - \sin n \theta  \,u^s_n(x)  . \]
Analogously we derive formula (\ref{eq:formula_gns_rotation}) for $\wh{g}_n^s$.
\end{proof}

 In order to relate \eqref{eq:original} and \eqref{eq:CGPTrotation_2} we need to have a better understanding of the boundary operator that we are integrating. We have already written an integral representation for the operator $\Lambda_1^{-1}\Lambda_{{\gamma}}^{{D}}: H^{-1/2}_0(\partial {D}) \arr H^{-1/2}_0(\partial {D})$ in the previous subsection (see \eqref{eq:integral_representation_original_setting}).
Now, we proceed similarly for the operator that plays a role in the rotated problem: 

\[\Lambda_{\wh{1}}^{-1}\Lambda_{\wh{\gamma}}^{\wh{D}}: H^{-1/2}_0(\partial \wh{D}) \arr H^{-1/2}_0(\partial \wh{D}) \]
\[ \wh{g}_n^c \longmapsto \frac{\partial  \wh{h}}{\partial \nu_{\wh{x}} } \biggr |_{-} ,\]
where $ \wh{h}$ is the solution to the boundary value problem
\begin{equation} \label{eq:bvphat_rotation} \begin{cases} \Laplacian \wh{h} = 0  &\quad \mbox{in } \wh{D}, \\ \nabla_{\wh{x}} \wh{h} \cdot {\nu_{\wh{x}}} =  \wh{g}^c_n & \quad \mbox{on } \partial \wh{D} .\end{cases} \end{equation}

Let $\wh{y} = \mathbf{R} y \in \wh{D}$ and consider the corresponding Neumann function $\wh{N}_1(\wh{x},\wh{y})$, that is,  the solution to
\begin{equation*} \label{eq:N_bvphat_rotation} \begin{cases} \Laplacian_{\wh{x}} \wh{N}_1 (\wh{x},\wh{y}) = - \delta_{\wh{y}}(\wh{x}) , &  \qquad \wh{x}\in \wh{D} , \\ \nabla_{\wh{x}} \wh{N}_1(\wh{x},\wh{y}) \cdot \nu_{\wh{x}} = - \frac{1}{|\partial \wh{D}|} , & \qquad \wh{x} \in \partial \wh{D}, \\ \displaystyle \int_{\partial \wh{D}} \wh{N}_1(\wh{x},\wh{y}) \dds_{\wh{x}} = 0. \end{cases} \end{equation*}
Exploiting the rotational invariance of the Laplacian gives
\begin{equation*} \label{eq:N_bvphat_rotation_2} \begin{cases} (\Laplacian_{{x}} \wh{N}_1 \circ \Phi_\mathbf{R}) (x, \mathbf{R}y) = - \delta_{\mathbf{R} y}(\mathbf{R}x)  = -  \delta_y (x) ,&  \qquad {x}\in {D}, \\   \mathbf{R} (\nabla_{{x}}\wh{N}_1 \circ \Phi_\mathbf{R} )( x ,\mathbf{R}y) \cdot (\mathbf{R} \nu_{{x}} )= - \frac{1}{|\partial {D}|} , & \qquad {x} \in \partial {D} ,\\ \displaystyle \int_{\partial {D}} \wh{N}_1(\mathbf{R} {x},\wh{y}) \dds_{{x}} = 0 .\end{cases} \end{equation*}
Therefore, 
\begin{equation*} \label{eq:N_bvphat_rotation_3} \begin{cases} (\Laplacian_{{x}} \wh{N}_1 \circ \Phi_{\mathbf{R}} ) ( x, \mathbf{R}y) = - \delta_{\mathbf{R}y}(\mathbf{R}x)  = -  \delta_y (x), &  \qquad {x}\in {D}, \\   (\nabla_{{x}}\wh{N}_1 \circ \Phi_{\mathbf{R}} )(  x , \mathbf{R}y) \cdot \nu_{{x}} = -\frac{1}{|\partial {D}|} , & \qquad {x} \in \partial {D}, \\ \displaystyle \int_{\partial {D}} \wh{N}_1(\mathbf{R} {x},\mathbf{R}  {y}) \dds_{{x}} = 0 .\end{cases} \end{equation*}

One can easily see that  $\wh{N}_1 ( \mathbf{R}x , \mathbf{R} y)$ and ${N}_1 (x ,  y)$ satisfy exactly the same boundary value problem \eqref{eq:N_bvp_transl}. The uniqueness of a solution to \eqref{eq:N_bvp_transl} yields
\begin{equation} \label{eq:identity_Nfc_rotation} \wh{N}_1 (\mathbf{R} x , \mathbf{R} y) = {N}_1 (x,y) \, .\end{equation}

Since the solution $\wh{h}$ to \eqref{eq:bvphat_rotation} can be represented by means of the Neumann function $\wh{N}_1$:
\[ \wh{h}(\wh{y}) = \int_{\partial \wh{D}} \wh{N}_1 (\wh{x} , \wh{y} )  \,\wh{g}_n^c(\wh{x}) \text{ d} s_{\wh{x}}\, ,  \]
we have, for $y \in \partial D$,
\[ \begin{split} \Lambda_{\wh{1}}^{-1} \Lambda_{\wh{\gamma}}^{\wh{D}} [\wh{g}_n^c] (\wh{y} )  & = \nabla_{\wh{y}}\,\wh{h}(\wh{y}) \cdot \nu_{\wh{y}} 
\\ &= \int_{\partial \wh{D}}  (\nabla_{\wh{y}} \, \wh{N}_1) (\wh{x} , \wh{y} ) \cdot \nu_{\wh{y}} \,\wh{g}_n^c(\wh{x}) \text{ d} s_{\wh{x}} 
\\ & = \int_{\partial {D}}  (\nabla_{\wh{y}} \, \wh{N}_1) ( \mathbf{R} x , \wh{y}) \cdot \nu_{\wh{y}} \, \wh{g}_n^c(\mathbf{R}x ) \,\text{ d} s_{{x}} \\ & = \int_{\partial {D}}  \mathbf{R} (\nabla_{y} \,  \wh{N}_1 ( \mathbf{R}x , \mathbf{R} \, \cdot \,) ) (\mathbf{R} {x} , y) \cdot \mathbf{R} \, \nu_{{y}} \, ( \cos n\theta \,g^c_n(x)  - \sin n \theta  \,g^s_n(x) )  \text{ d} s_{{x}} 
 \\ & = \int_{\partial {D}}  (\nabla_{y} \, N_1 ) ({x} , y) \cdot \nu_{y} \, ( \cos n \theta \,g^c_n(x)  - \sin n \theta  \,g^s_n(x) )  \text{ d} s_{x}
\\ & = \cos n \theta \, \int_{\partial {D}}  (\nabla_{y} \, N_1 ) ({x} , y) \cdot \nu_{y}  \,g^c_n(x)    \text{ d} s_{x}  - \sin n \theta \,\int_{\partial {D}}  (\nabla_{y} \, N_1 ) ({x} , y) \cdot \nu_{y}  \,g^s_n(x)  \text{ d} s_{x} \,. \end{split} \]
Here, we have made the change of variables $\wh{x} = \mathbf{R} x$ and used  identity \eqref{eq:identity_Nfc_rotation}. Therefore, it follows that 
\begin{equation} \label{eq:lambda_formula_rotation} \Lambda_{\wh{1}}^{-1} \Lambda_{\wh{\gamma}}^{\wh{D}} [\wh{g}_n^c] (\mathbf{R} x )  = \cos n\theta \, \Lambda_1^{-1}\Lambda_{{\gamma}}^{{D}} [g_n^c] (x) - \sin n \theta \, \Lambda_1^{-1}\Lambda_{{\gamma}}^{{D}} [g_n^s] (x)  \,. \end{equation}
Hence, from Lemma \ref{lemma:gnc_rotation} and \eqref{eq:lambda_formula_rotation},  we get
\[ \begin{split} \wh{g}_n^c(\mathbf{R} x ) -  \Lambda_{\wh{1}}^{-1} \Lambda_{\wh{\gamma}}^{\wh{D}} [\wh{g}_n^c] (\mathbf{R} x ) & =  \cos n \theta \,g^c_n(x)  - \sin n \theta  \,g^s_n(x) -  \cos n \theta \, \Lambda_1^{-1}\Lambda_{{\gamma}}^{{D}} [g_n^c] (x) + \sin n \theta \, \Lambda_1^{-1}\Lambda_{{\gamma}}^{{D}} [g_n^s] (x)\\ & = \cos n \theta \, ( g^c_n(x) - \Lambda_1^{-1}\Lambda_{{\gamma}}^{{D}} [g_n^c] (x) ) - \sin n \theta  \, ( g^s_n(x) -  \Lambda_1^{-1}\Lambda_{{\gamma}}^{{D}} [g_n^s] (x)  ) \,.\end{split} \]
\[\begin{split} M^{cc}_{mn} (\wh{\gamma}, \wh{D}) & =    \int_{\partial D}  r_x^m \cos (m (\theta_x + \theta))\, (  \wh{g}_n^c(\mathbf{R} x) -   \Lambda_{\wh{1}}^{-1}\Lambda_{\wh{\gamma}}^{\wh{D}} [ \wh{g}_n^c](\mathbf{R} x) ) \text{ d} s_{x} 
\\ &=  \int_{\partial D} [ r^m_x \cos ( m \theta_x ) \cos ( m \theta )  - r^m_x \sin ( m \theta_x ) \sin ( m \theta ) ] (  \wh{g}_n^c(\mathbf{R} x) -   \Lambda_{\wh{1}}^{-1}\Lambda_{\wh{\gamma}}^{\wh{D}} [ \wh{g}_n^c](\mathbf{R} x) )  \text{ d} s_{x}
\\ &= \cos ( m \theta )  \cos ( n\theta) \,  \int_{\partial D} r^m_x \cos ( m \theta_x )  (  g_n^c(x) -   \Lambda_{1}^{-1}\Lambda_{\gamma}^{D} [g_n^c]( x) ) \text{ d} s_{x}
\\ & - \cos ( m \theta )  \sin ( n\theta ) \,  \int_{\partial D} r^m_x \cos ( m \theta_x )  (  g_n^s(x) -   \Lambda_{1}^{-1}\Lambda_{\gamma}^{D} [g_n^s]( x) ) \text{ d} s_{x} 
\\&-  \sin ( m \theta ) \cos ( n \theta )\int_{\partial D}   r^m_x \sin ( m \theta_x )   (  g_n^c(x) -   \Lambda_{1}^{-1}\Lambda_{\gamma}^{D} [g_n^c]( x) ) \text{ d} s_{x}
\\&+  \sin ( m \theta ) \sin ( n \theta )\int_{\partial D}   r^m_x \sin ( m \theta_x )   (  g_n^s(x) -   \Lambda_{1}^{-1}\Lambda_{\gamma}^{D} [g_n^s]( x) ) \text{ d} s_{x}
\\ & = \cos ( m \theta )  \cos ( n\theta)   M^{cc}_{mn}  - \cos ( m \theta )  \sin ( n\theta ) M^{cs}_{mn} \\ & -  \sin ( m \theta ) \cos ( n \theta ) M^{sc}_{mn} + \sin ( m \theta ) \sin ( n \theta ) M^{ss}_{mn} \,.
  \end{split}\]

Similar computations lead to the rotation formulas for the others CGPTs $\wh{M}^{cs}_{mn}, \wh{M}^{sc}_{mn}$ and $\wh{M}^{ss}_{mn}$. All these formulas can be written in a matrix form:
\[ { \begin{bmatrix} \wh{M}^{cc}_{mn} &  \wh{M}^{sc}_{mn} \\  \wh{M}^{cs}_{mn}  &  \wh{M}^{ss}_{mn}  \end{bmatrix} = \mathbf{R}(n \theta) \cdot \begin{bmatrix} {M^{cc}_{mn}} &  {M^{sc}_{mn}} \\  {M^{cs}_{mn}}  &  {M^{ss}_{mn}}  \end{bmatrix} \cdot \mathbf{R}(m \theta)^T  \, ,} \]
where $\mathbf{R}(\theta)$ is defined in \eqref{addmatrix}.

\subsection{Scaling formula}
Similarly to what we have done for translations and rotations we want to investigate how $M_{mn}$ changes with respect to a scaling of $D$. 

\medskip

Let $s > 0$ and assume that $z=0$. We denote $\wh{x} = s x$, $\wh{D}:= s D$  and $\wh{\gamma}(\wh{x}):= \gamma (x)$. We want to relate $M^{cc}_{mn} (\gamma , D)$ with $ M^{cc}_{mn} (\wh{\gamma} , \wh{D})$ given by

\begin{equation*}\label{eq:CGPTscaling} M^{cc}_{mn} (\wh{\gamma}, \wh{D}) = \wh{M}_{mn} =  \int_{\partial \wh{D}}  \text{Re}(\wh{x}^m)  (  \wh{g}_n^c(\wh{x}) -   \Lambda_{\wh{1}}^{-1}\Lambda_{\wh{\gamma}}^{\wh{D}} [ \wh{g}_n^c](\wh{x}) ) \text{ d} s_{\wh{x}} \,. \end{equation*}
By the change of variables $\wh{x} = sx$, we obtain
\begin{equation} \label{eq:CGPTscaling_2} M^{cc}_{mn} (\wh{\gamma} , \wh{D}) = \wh{M}_{mn} =  s^{m+1} \, \int_{\partial D}  \text{Re}(x^m)  (  \wh{g}_n^c(sx) -   \Lambda_{\wh{1}}^{-1}\Lambda_{\wh{\gamma}}^{\wh{D}} [ \wh{g}_n^c](sx) ) \text{ d} s_{x} . \end{equation}
\begin{lem}\label{lemma:gnc_scaling} We have
\begin{equation} \label{eq:formula_gnc_scaling} \wh{g}_n^c(s x) = s^{n -1} g_n^c (x) \,, \end{equation} 
and 
\begin{equation} \label{eq:formula_gns_scaling}  \wh{g}_n^s(s x) = s^{n -1} g_n^s (x) \,. \end{equation}
\end{lem}

\begin{proof} We show the first identity. The one for $\wh{g}_n^s$ can be proved in the same way. Let $\wh{u}^c_n$ be the solution to
\begin{equation} \label{eq:lemma_gnc_scaling}\begin{cases} \nabla_{\wh{x}} \cdot \wh{\gamma}(\wh{x}) \nabla_{\wh{x}} \wh{u}^c_n (\wh{x}) = 0 &\qquad \mbox{in } \RR^2 , \\  \wh{u}^c_n (\wh{x}) - r^{n}_{\wh{x}} \cos (n \theta_{\wh{x}}) = O(r_{\wh{x}}^{-1}) &  \qquad \mbox{as } |\wh{x}| \to + \infty ,\end{cases} \end{equation}
Then $\wh{g}_n^c:= \wh{\gamma} \dfrac{\partial \wh{u}^c_n}{\partial \nu_{\wh{x}}}$. By a change of variables in \eqref{eq:lemma_gnc_scaling} and by setting  $v^c_n(x):= \wh{u}^c_n  (s x)$, we obtain
\[ \label{eq:lemma_gnc_scaling_2} \begin{cases} \nabla_{{x}} \cdot {\gamma}({x}) \nabla_{{x}} v^c_n  (x) = 0 &\qquad \mbox{in } \RR^2 ,\\ s ( {v}^c_n ({x}) - s^{n}  r^{n}_{{x}} \cos (n \theta_{{x}}) )= O(r_{{x}}^{-1}) &  \qquad \mbox{as } |\wh{x}| \to + \infty  .\end{cases}
\]
Therefore, $s^n {v}^c_n ({x})$ solves the same problem as $u^c_n(x)$. By the uniqueness of a solution,  we get
\[ s^{-n} {v}^c_n ({x}) = u^c_n(x) .\]
So
\[  \nabla_x u^c_n(x) = s^{-n} \nabla_x {v}^c_n ({x})  = s^{-n+1} \nabla_{\wh{x}}  \wh{u}^c_n (\wh{x})\,. \]
Hence,
\[ g^c_n(x) = s^{-n+1}  {\wh{g}}^c_n(\wh{x}).\]
\end{proof}
 
We want to relate \eqref{eq:original} and \eqref{eq:CGPTscaling_2}. We refer to \eqref{eq:integral_representation_original_setting} for an integral representation of the operator $\Lambda_1^{-1}\Lambda_{{\gamma}}^{{D}}$ .

\bigskip

Now we proceed similarly for the operator that plays a role in the scaled problem: \[ \Lambda_{\wh{1}}^{-1}\Lambda_{\wh{\gamma}}^{\wh{D}}: H^{-1/2}_0(\partial \wh{D}) \arr H^{-1/2}_0(\partial \wh{D}) \]
\[ \wh{g}_n^c \longmapsto \frac{\partial  \wh{h}}{\partial \nu_{\wh{x}} } \biggr |_{-} ,\]
where $ \wh{h}$ is the solution to the boundary value problem
\begin{equation} \label{eq:bvphat_scaling} \begin{cases} \Laplacian \wh{h} = 0 &\quad \mbox{in } \wh{D} ,\\ \nabla_{\wh{x}} \wh{h} \cdot {\nu_{\wh{x}}} =  \wh{g}_n^c  &\quad \mbox{on } \partial \wh{D} \,.\end{cases}\end{equation}

Let $\wh{y} = sy \in \wh{D}$ and consider the corresponding Neumann function $\wh{N}_1(\wh{x},\wh{y})$, that is,  the solution to
\[\label{eq:N_bvphat_scaling} \begin{cases} \Laplacian_{\wh{x}} \wh{N}_1 (\wh{x},\wh{y}) = - \delta_{\wh{y}}(\wh{x}) , &  \qquad \wh{x}\in \wh{D} , \\ \nabla_{\wh{x}} \wh{N}_1(\wh{x},\wh{y}) \cdot \nu_{\wh{x}} = -\frac{1}{|\partial \wh{D}|} ,  & \qquad \wh{x} \in \partial \wh{D} , \\ \displaystyle \int_{\partial \wh{D}} \wh{N}_1(\wh{x},\wh{y}) \dds_{\wh{x}} = 0 . \end{cases} \]
Then, 
\[\label{eq:N_bvphat_scaling_2} \begin{cases} \frac{1}{s^2} \,\Laplacian_{{x}} \wh{N}_1 (s x,s y) = - \delta_{sy}(sx)  = - \delta_{0} (s ( x - y )) = - \frac{1}{s^2} \delta_y (x),  &  \qquad {x}\in {D} , \\ \frac{1}{s} \,\nabla_{{x}} \wh{N}_1( s x ,s y) \cdot \nu_{{x}} = -\frac{1}{s |\partial {D}|} , & \qquad {x} \in \partial {D}, \\ \displaystyle \int_{\partial {D}} \wh{N}_1(s {x},\wh{y}) \dds_{{x}} = 0 , \end{cases}\]
which shows that
\[ \label{eq:N_bvphat_scaling_3} \begin{cases} \Laplacian_{{x}} \wh{N}_1 (s x,s y) = -  \delta_y (x), &  \qquad {x}\in {D}, \\ \nabla_{{x}} \wh{N}_1( s x ,s y) \cdot \nu_{{x}} =  -\frac{1}{|\partial {D}|},  & \qquad {x} \in \partial {D}, \\ \displaystyle \int_{\partial {D}} \wh{N}_1(s {x},s {y}) \dds_{{x}} = 0 .\end{cases} \]
One can easily see that  $\wh{N}_1 ( s x , s y)$ and ${N}_1 (x ,  y)$ satisfy exactly the same boundary value problem \eqref{eq:N_bvp_transl}. Therefore, 
\begin{equation} \label{eq:identity_Nfc_scaling} \wh{N}_1 (s x , s y) = {N}_1 (x,y) \, .\end{equation}
The solution $ \wh{h}$ to  \eqref{eq:bvphat_scaling} can be represented by means of the Neumann function $\wh{N}_1$:
\[ \wh{h}(\wh{y}) = \int_{\partial \wh{D}} \wh{N}_1 (\wh{x} , \wh{y} )  \,\wh{g}_n^c(\wh{x}) \text{ d} s_{\wh{x}} \,, \]
and so, for $y \in \partial D$, we have
\[ \begin{split} \Lambda_{\wh{1}}^{-1} \Lambda_{\wh{\gamma}}^{\wh{D}} [\wh{g}] (\wh{y} )  & = \nabla_{\wh{y}}\,\wh{h}(\wh{y}) \cdot \nu_{\wh{y}} 
\\ &= \int_{\partial \wh{D}}  (\nabla_{\wh{y}} \, \wh{N}_1) (\wh{x} , \wh{y} ) \cdot \nu_{\wh{y}} \,\wh{g}_n^c(\wh{x}) \text{ d} s_{\wh{x}} 
\\ & = \int_{\partial {D}}  (\nabla_{\wh{y}} \, \wh{N}_1) ( s x , \wh{y}) \cdot \nu_{{y}} \, s^{n-1} \,{g}_n^c({x}) \, s \text{ d} s_{{x}} \\ & = s^{n-1} \,  \int_{\partial {D}}  (\nabla_{y} \,  \wh{N}_1 ( s x , s \, \cdot \,) ) (s {x} , y) \cdot \nu_{{y}} \,{g}_n^c({x}) \text{ d} s_{{x}} 
 \\ & = s^{n-1}\,\int_{\partial {D}}  (\nabla_{y} \, N_1 ) ({x} , y) \cdot \nu_{{y}} \,{g}_n^c({x}) \text{ d} s_{{x}} . \end{split} \]
Here, we have made the change of variables $\wh{x} = s x$ and used  identity \eqref{eq:identity_Nfc_scaling}. Therefore, it follows that 
\[ \Lambda_{\wh{1}}^{-1} \Lambda_{\wh{\gamma}}^{\wh{D}} [\wh{g}] (\wh{y} )  = s^{n-1} \, \Lambda_1^{-1}\Lambda_{{\gamma}}^{{D}} [g] (y) . \]
Then
\[  \begin{split} M^{cc}_{mn} (\wh{\gamma} , \wh{D})  & =  s^{m+1} \,  \int_{\partial D}  \text{Re}(x^m)  (  s^{n-1} {g_n^c}(x) -   s^{n - 1} \Lambda_1^{-1}\Lambda_{\gamma}^{D} [ g_n^c](x) ) \text{ d} s_{x}
\\ & =  s^{m+n} \,  \int_{\partial D}  \text{Re}(x^m)  ( {g_n^c}(x) -  \Lambda_1^{-1}\Lambda_{\gamma}^{D} [ g_n^c](x) ) \text{ d} s_{x}
\\ & =  s^{m+n} \,  M^{cc}_{mn} (\gamma , D) .\end{split}\]
Hence, we obtain the following scaling formula
\[ { M^{cc}_{mn} (\wh{\gamma} , \wh{D})  = s^{m+n} \,  M^{cc}_{mn} (\gamma , D)  \,.} \]
Analogously, we get
\[ M^{cs}_{mn} (\wh{\gamma} , \wh{D})  = s^{m+n} \,  M^{cs}_{mn} (\gamma , D) , \] \[ M^{sc}_{mn} (\wh{\gamma} , \wh{D})  = s^{m+n} \,  M^{sc}_{mn} (\gamma , D),\] \[ M^{ss}_{mn} (\wh{\gamma} , \wh{D})  = s^{m+n} \,  M^{ss}_{mn} (\gamma , D) .\]

\bigskip

 In order to simplify the notation, for any pair of indices $m, n$, we denote by $\mathbf{M}_{mn}:=\begin{bmatrix} {M^{cc}_{mn}} &  {M^{sc}_{mn}} \\  {M^{cs}_{mn}}  &  {M^{ss}_{mn}}  \end{bmatrix}$ and $\wh{\mathbf{M}}_{mn}:= \begin{bmatrix} \wh{M}^{cc}_{mn} &  \wh{M}^{sc}_{mn} \\  \wh{M}^{cs}_{mn}  &  \wh{M}^{ss}_{mn}  \end{bmatrix}$. We introduce also the following notation:
\begin{itemize}
\item $T_z D = \{ x + z , \text{ for } x \in D \}$, $(T_z \star \gamma)(x) = \gamma(x-z)$, for $z \in \RR^2$;
\item[]
\item $R_\theta D = \{ e^{i \theta} x , \text{ for } x \in D \}$, $(R_\theta \star \gamma)(x) = \gamma(e^{-i \theta} x)$, for $\theta \in [0, 2 \pi )$;
\item[]
\item $s D = \{ s x , \text{ for } x \in D \}$, $(s \star \gamma)(x) = \gamma(s^{-1} x)$, for $s > 0$, 
\end{itemize}
where $D$ is an open set and $\gamma$ is a conductivity distribution in the plane.

\medskip

We summarize the formulas that we obtained so far in the following theorem.

\begin{prop} \label{prop:cgpt_transf_formulas} For any pair $m,n$ of indices, $m,n = 1, 2 , \ldots$, the following transformation formulas hold true:
\[ \begin{split}  {\mathbf{M}_{mn}} (T_z \star \gamma, T_z D) &= \sum_{k=1}^m \sum_{r=1}^n r_z^{m-k}  r_z^{n-r}  \left ( \begin{matrix} m \\ k \end{matrix} \right ) \left ( \begin{matrix} n\\ r \end{matrix} \right ) \mathbf{R}((n-r)\theta_z)   {\mathbf{M}_{kr}}(\gamma, D)  \mathbf{R}((m-k)\theta_z)^T ,  \\
 {\mathbf{M}_{mn}} (R_\theta \star \gamma, R_\theta D)  & = \mathbf{R}(n \theta)  {\mathbf{M}_{mn}}(\gamma, D)  \mathbf{R}(m \theta)^T , \end{split}\]
 and 

\medskip
 \(
  {\mathbf{M}_{mn}}(s \star \gamma, s D)  = s^{m+n} \,  {\mathbf{M}_{mn}}(\gamma, D) .  \)
\end{prop}
\subsection{Complex CGPTs}

As observed in \cite{Am3}, it is convenient to consider complex combinations of CGPTs. For a pair of indices $m, n = 1 , 2 , ...$ ,  we introduce the following quantities
\begin{equation} \label{eq:CCGPTs} \begin{matrix} \mathbf{N}_{mn}^{(1)} (\gamma , D) = (M_{mn}^{cc} - M_{mn}^{ss}) + i  (M_{mn}^{cs} + M_{mn}^{sc}),  \\ \mathbf{N}_{mn}^{(2)} (\gamma , D) = (M_{mn}^{cc} + M_{mn}^{ss}) + i  (M_{mn}^{cs} - M_{mn}^{sc}) . \end{matrix}  \end{equation}

Using relations of Proposition \ref{prop:cgpt_transf_formulas} it is straightforward to prove similar rules than those derived in \cite{Am2} for the complex CGPTs \eqref{eq:CCGPTs}. 
\begin{prop} \label{prop:N_cgpts_properties} For all integers $m , n$, and geometric parameters $\theta , s$ and $z$, the following holds:
\begin{equation} \label{eq:N_rot}  \mathbf{N}^{(1)}_{mn} (R_\theta \star \gamma , R_\theta D) = e^{i (m + n)} \mathbf{N}^{(1)}_{mn} (\gamma, D)  , \qquad  \mathbf{N}^{(2)}_{mn} (R_\theta \star \gamma , R_\theta D) = e^{i (n -  m)} \mathbf{N}^{(2)}_{mn} (\gamma, D) , \end{equation}
\begin{equation} \label{eq:N_scal}  \mathbf{N}^{(1)}_{mn}  (s \star \gamma, s D)= s^{m+n} \mathbf{N}^{(1)}_{mn} (\gamma, D)  , \qquad  \mathbf{N}^{(2)}_{mn} (s \star \gamma, s D) = s^{m+n} \mathbf{N}^{(2)}_{mn} (\gamma, D) ,\end{equation}
and 
\begin{equation} \label{eq:N_transl}  \mathbf{N}^{(1)}_{mn} (T_z \star \gamma, T_z D) = \sum_{l = 1}^m   \sum_{k= 1}^n \mathbf{C}^z_{ml}  \mathbf{N}^{(1)}_{lk} (\gamma, D) \mathbf{C}^z_{nk}  , \qquad  \mathbf{N}^{(2)}_{mn} (T_z \star \gamma , T_z D) = \sum_{l = 1}^m   \sum_{k= 1}^n \overline{\mathbf{C}^z_{ml}} \mathbf{N}^{(2)}_{lk} (\gamma, D) \mathbf{C}^z_{nk}   , \end{equation}
where $\mathbf{C}^z$ is a lower triangular matrix with the $m,n$-th entry given by
\[ \mathbf{C}^z_{mn} = \binom{m}{n} \, z^{m-n} .\]
\end{prop}

We define the complex CGPT matrices by
\[ \mathbf{N}^{(1)}: = ( \mathbf{N}^{(1)}_{mn} )_{m,n}, \qquad \mathbf{N}^{(2)}: = ( \mathbf{N}^{(2)}_{mn} )_{m,n} .  \]
Setting $w = s e^{i \theta}$ we introduce the diagonal matrix $\mathbf{G}^w$ with $m$-th diagonal entry given by $s^m e^{i m \theta}$.

Applying one after the other the properties of Proposition \ref{prop:N_cgpts_properties} we immediately get the following relations:
\begin{equation}\label{eq:relation_1}
\mathbf{N}^{(1)} (T_z \star (s \star (R_\theta \star \gamma)), T_z s R_\theta D) = \mathbf{C}^z \mathbf{G}^w \mathbf{N}^{(1)}(\gamma, D) \mathbf{G}^w (\mathbf{C}^z)^T\, ,
\end{equation}
\begin{equation} \label{eq:relation_2}
\mathbf{N}^{(2)} (T_z \star (s \star (R_\theta \star \sigma)) , T_z s R_\theta D) = \overline{ \mathbf{C}^z \mathbf{G}^w} \mathbf{N}^{(2)} (\gamma, D) \mathbf{G}^w (\mathbf{C}^z)^T \,.
\end{equation}
Relations \eqref{eq:relation_1} and \eqref{eq:relation_2} still hold for the truncated CGPTs of finite order, due to the triangular shape of the matrix $\mathbf{C}^z$.

\bigskip

We call a dictionary $\mathcal{D}$ a collection of pairs $(\sigma, B)$, where $B$ is a standard shape centered at the origin, with characteristic size of order $1$, and $\sigma$ is a conductivity distribution such that $\text{supp}(\sigma -1) = \overline{B}$.  

We assume that  a reference dictionary $\mathcal{D}$ is initially given. Furthermore, suppose to consider a pair $(\gamma, D)$, which is unknow, that is obtained from an element $(\sigma, B) \in \mathcal{D}$ by applying some unknown rotation $\theta$, scaling $s$ and translation $z$, i.e., $D = T_z s R_\theta \,B$ and $\gamma = T_z \star (s \star (R_\theta \star \sigma))$.

\subsection{Conductivity descriptors}

If $D = T_z s R_\theta \,B$ and $\gamma = T_z \star (s \star (R_\theta \star \sigma))$  then the following identities hold true:
\begin{equation}\mathbf{N}_{11}^{(1)} (\gamma, D) = w^2 \mathbf{N}_{11}^{(1)} (\sigma, B), \end{equation}
\begin{equation}\mathbf{N}_{12}^{(1)} (\gamma , D) = 2 \mathbf{N}_{11}^{(1)}(\gamma , D) z +  w^3 \mathbf{N}_{12}^{(1)} (\sigma,B), \end{equation}
\begin{equation}\label{eq:identity_3} \mathbf{N}_{11}^{(2)} (\gamma , D) =s^2 \mathbf{N}_{11}^{(2)} (\sigma, B), \end{equation}
\begin{equation}\label{eq:identity_4} \mathbf{N}_{12}^{(2)} (\gamma , D) = 2 \mathbf{N}_{11}^{(2)}(\gamma , D) z +  s^2 w \mathbf{N}_{12}^{(2)} (\sigma,B), \end{equation}
where $w = s e^{i \theta}$.

\medskip

From identities \eqref{eq:identity_3} and \eqref{eq:identity_4} we obtain the relation:
\begin{equation} \frac{\mathbf{N}_{12}^{(2)} (\gamma, D)}{2 \mathbf{N}_{11}^{(2)} (\gamma, D)} = z + s e^{i \theta} \frac{\mathbf{N}_{12}^{(2)} (\sigma, B)}{2 \mathbf{N}_{11}^{(2)} (\sigma, B)} \,.\end{equation}

\medskip

Following \cite{Am3}, let $u = \frac{\mathbf{N}_{12}^{(2)} (\gamma, D)}{2 \mathbf{N}_{11}^{(2)} (\gamma, D)}$.
 We define the following quantities
\begin{equation} \mathcal{J}^{(1)} (\gamma, D) = \mathbf{N}^{(1)} ( T_{-u} \star \gamma, T_{-u} D ) = \mathbf{C}^{-u} \mathbf{N}^{(1)} (\gamma, D) ( \mathbf{C}^{-u})^T , \end{equation}
\begin{equation} \mathcal{J}^{(2)} (\gamma,D) = \mathbf{N}^{(2)} (T_{-u} \star \gamma , T_{-u} D ) = \overline{\mathbf{C}^{-u}} \mathbf{N}^{(2)} (\gamma, D) ( \mathbf{C}^{-u})^T , \end{equation}
where the matrix $\mathbf{C}^{-u}$ has been previously defined in Proposition \ref{prop:N_cgpts_properties}. These quantities are translation invariant.

\medskip

From $\mathcal{J}^{(1)}(\gamma,D) = (\mathcal{J}_{mn}^{(1)}(\gamma,D) )_{m,n}$, $\mathcal{J}^{(2)} (\gamma,D)  = (\mathcal{J}_{mn}^{(2)} (\gamma,D) )_{m,n}$, for each pair of indices $m,n$, we define  the scaling invariant quantities:
\begin{equation} \mathcal{S}^{(1)}_{mn}(\gamma,D)  = \frac{ \mathcal{J}_{mn}^{(1)}(\gamma,D)}{(\mathcal{J}_{mm}^{(2)}(\gamma,D)  \mathcal{J}_{nn}^{(2)}(\gamma,D) )^{1/2}} , \qquad \mathcal{S}^{(2)}_{mn} (\gamma, D) = \frac{ \mathcal{J}_{mn}^{(2)}(\gamma,D) }{(\mathcal{J}_{mm}^{(2)}(\gamma,D)  \mathcal{J}_{nn}^{(2)}(\gamma,D) )^{1/2}} . \end{equation}

Finally, we introduce the CGPT-based shape descriptors $\mathcal{I}^{(1)} = ( \mathcal{I}^{(1)}_{mn})_{m,n}$ and $\mathcal{I}^{(2)} = ( \mathcal{I}^{(2)}_{mn})_{m,n}$:

\[\mathcal{I}^{(1)}_{mn} = | \mathcal{S}^{(1)}_{mn} (\gamma,D)  | , \qquad \mathcal{I}^{(2)}_{mn} = | \mathcal{S}^{(2)}_{mn} (\gamma,D)  | , \]
where $| \, \cdot \, |$ denotes the modulus of a complex number. It is clear, by construction, that $\mathcal{I}^{(1)}$ and $\mathcal{I}^{(2)}$ are invariant under translation, rotation, and scaling.

\bigskip

The matching algorithm we refer to is rather simple, see Algorithm \ref{matching_algorithm}. This approach has been presented previously by Habib Ammari et al. in \cite{Am2}, where  shape descriptors have been exploited for dealing with homogeneous conductivities.

\bigskip

\begin{algorithm}[H]
\Input{the first $k$-th order shape descriptors $\mathcal{I}^{(1)}(D)$, $\mathcal{I}^{(2)}(D)$ of an unknown shape $D$.}

\nl \For{$B_n \in \mathcal{D}$}{
\nl  $e_n \leftarrow \left ( \| \mathcal{I}^{(1)}(B_n) - \mathcal{I}^{(1)}(D) \|_F^2 + \| \mathcal{I}^{(2)}(B_n) - \mathcal{I}^{(2)}(D) \|_F^2 \right )^{1/2}$ \;
\nl $n \leftarrow n +1$\;
 }
\Output{the true dictionary element $n^* \leftarrow \text{argmin}_n e_n$.}
 \caption{Shape identification based on transform invariant descriptors}
\label{matching_algorithm}
\end{algorithm}
$\| \cdot \|_F$ denotes the Frobenius norm of matrices.
\begin{rem} It is easy to see that all the radially symmetric conductivities possess the same conductivity descriptors $\mathcal{I}^{(1)}$ and $\mathcal{I}^{(2)}$. This is a consequence of the following identities:
\[\begin{split} M_{mn}^{cs} &= M_{mn}^{sc} = 0\quad \text{ for all } m,n,\\
 M_{mn}^{cc} &= M_{mn}^{ss}  = 0\quad \text{ if } m \ne n ,\\
 M_{mm}^{cc} &= M_{mm}^{ss}  \quad \text{ if } m = n . \end{split}\]
\end{rem}

\section{Numerical results}  \label{sec3}
In this section, we show some proof-of-concept numerical simulations about the dictionary-matching approach.
Henceforth, we will restrict ourselves to piecewise constant distributions only.

\subsection{Setting} \label{subsec-setting}
 Let $\mathcal{D}$ be the dictionary containing 10 standard conductivity distributions, as illustrated in \figref{smalldico}. Each one of the 5 shapes in the row a is equipped with homogeneous conductivity having parameter $k = 2$ (Triangle, Ellipse, Bean, Shield and Triangular Shield) whereas each coated shape in the row b is equipped with an inhomogeneous conductivity distribution having value $k_1 = 2$ in the outer coating and having value $k_2 = 4$ in the inner coating. All the shapes have the same characteristic size, which is of order one.

\begin{figure}[H]
	\centering
\begin{tabular}{cccccc}
& 1 & 2  & 3 & 4 & 5 \\
& & & & & \\
a \quad &  \raisebox{7ex - \height}{\includegraphics[ scale=0.23]{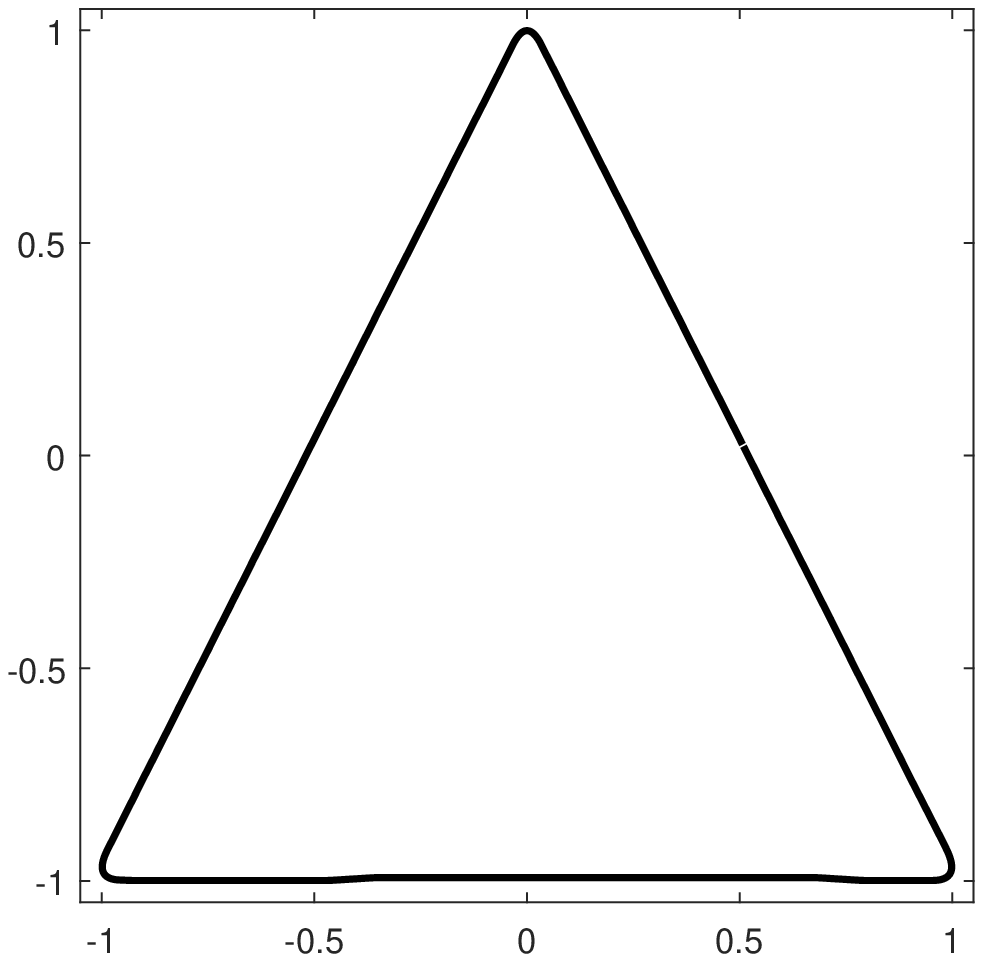}}   &
   \raisebox{7ex - \height}{\includegraphics[ scale=0.23]{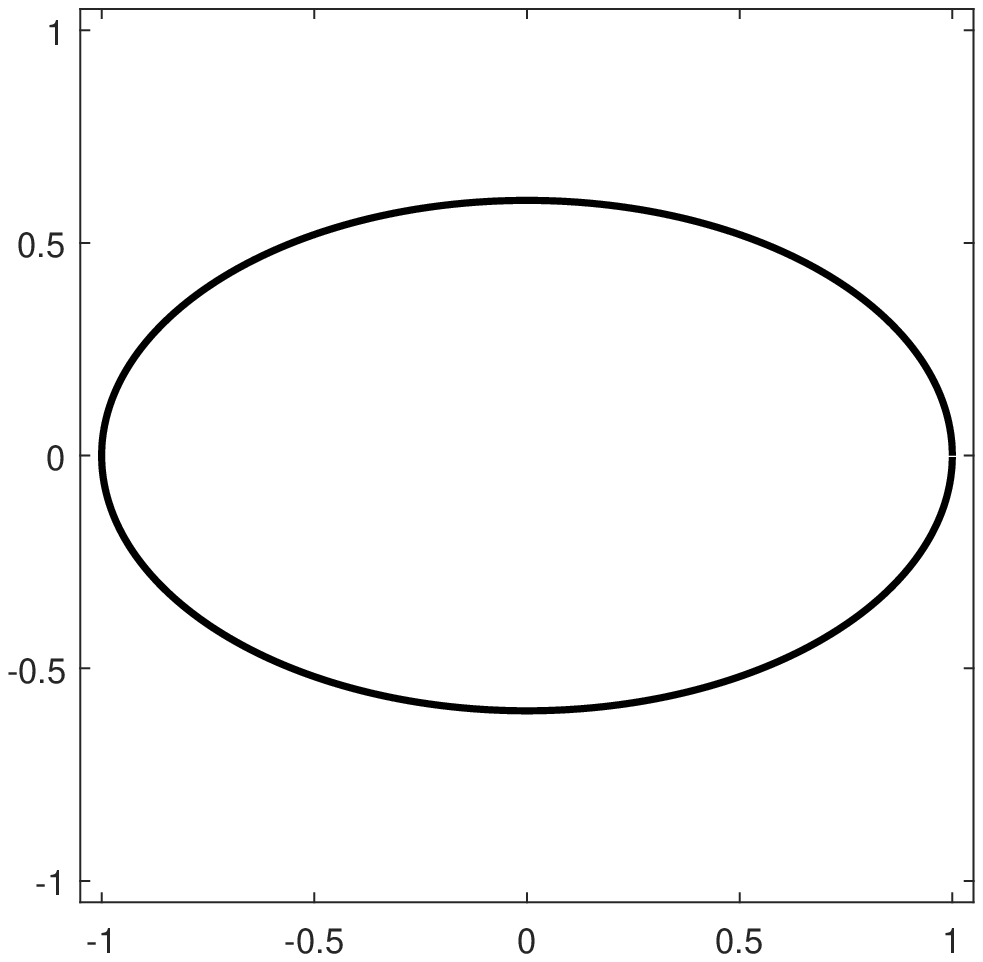}} & 
   \raisebox{7ex - \height}{\includegraphics[ scale=0.23]{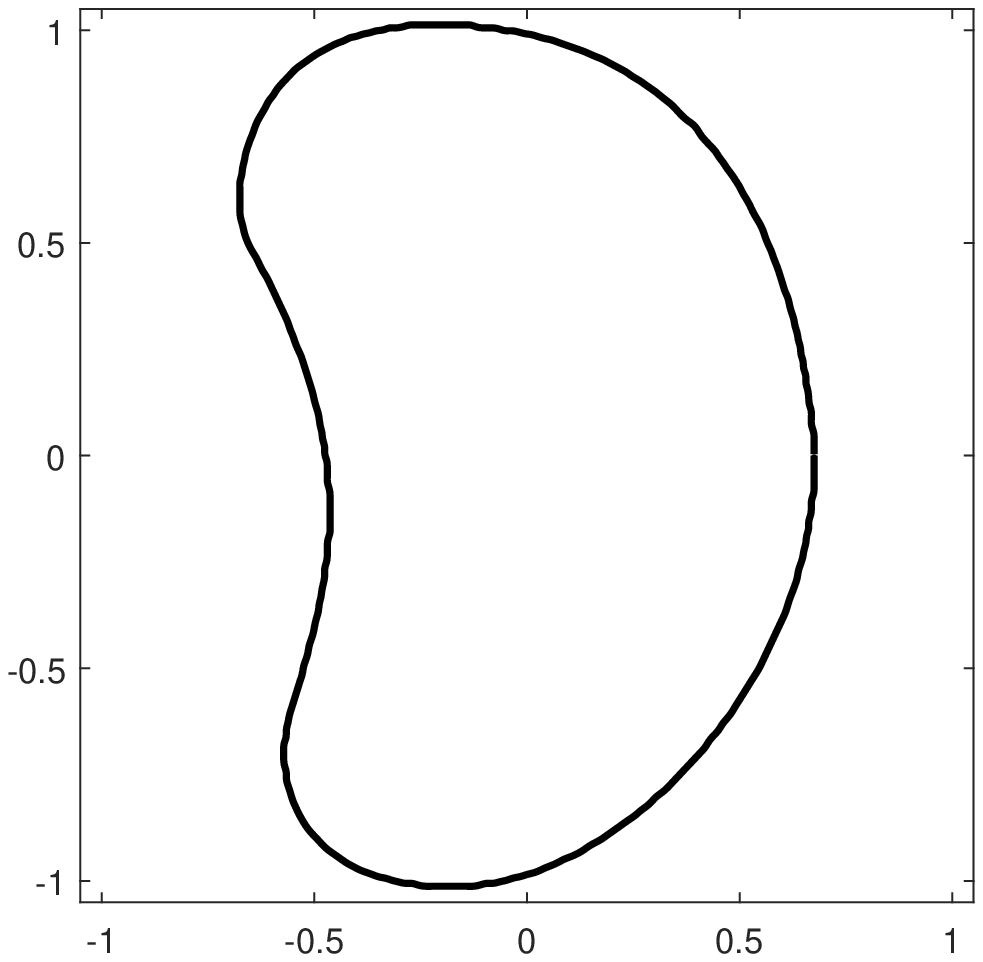}} &
   \raisebox{7ex - \height}{\includegraphics[ scale=0.23]{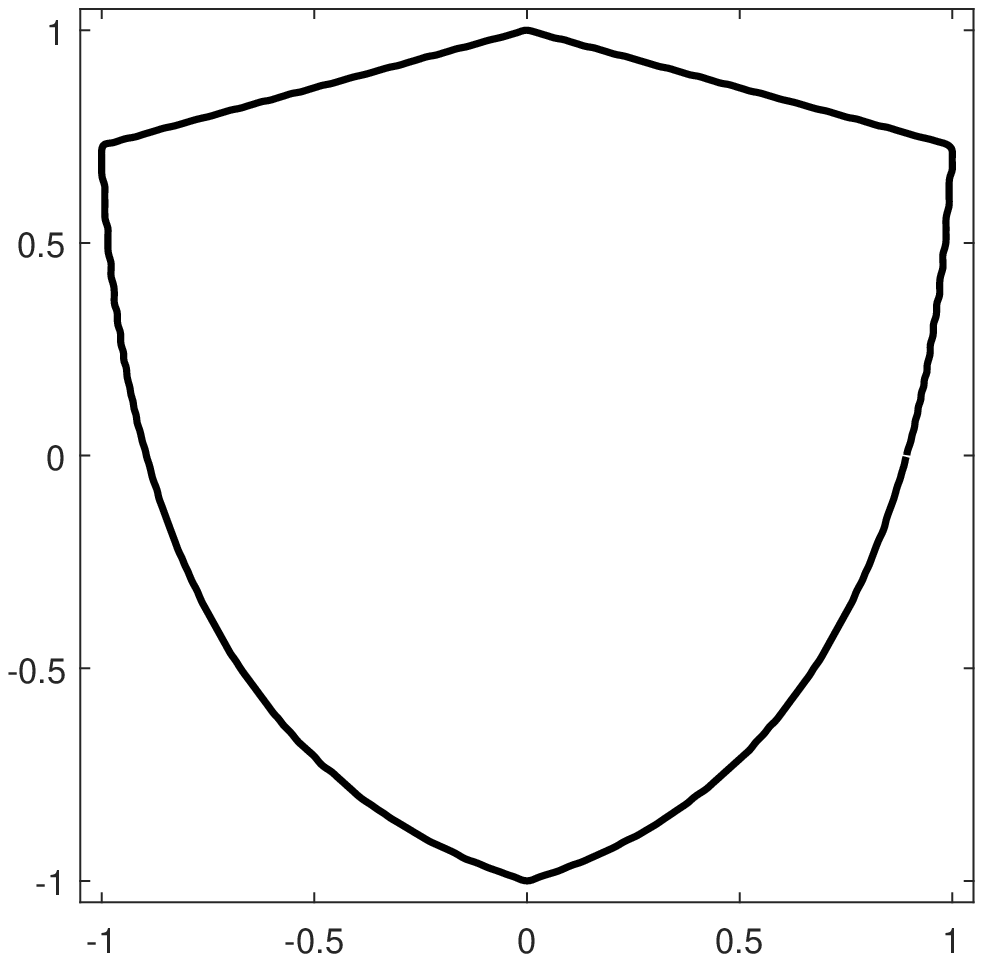}} &
   \raisebox{7ex - \height}{\includegraphics[ scale=0.23]{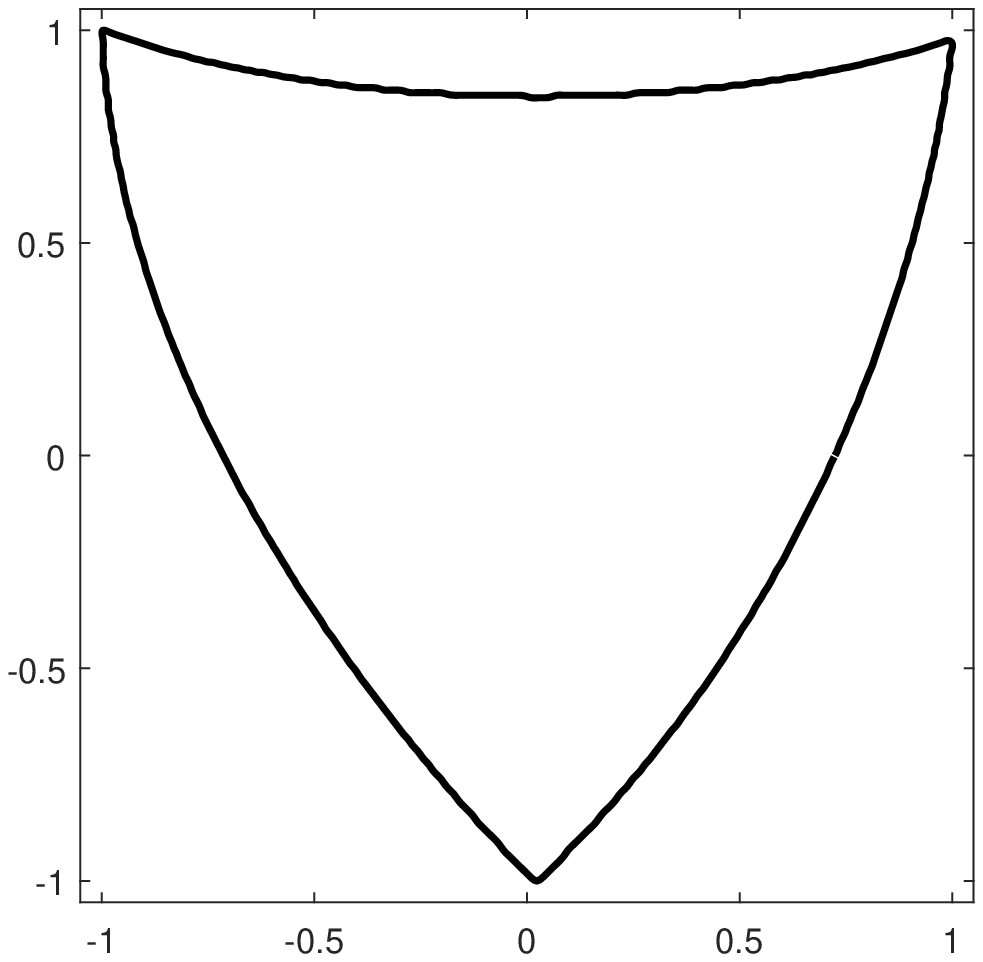}} \\
&&&&&\\
 b \quad &     \raisebox{7ex - \height}{\includegraphics[ scale=0.23]{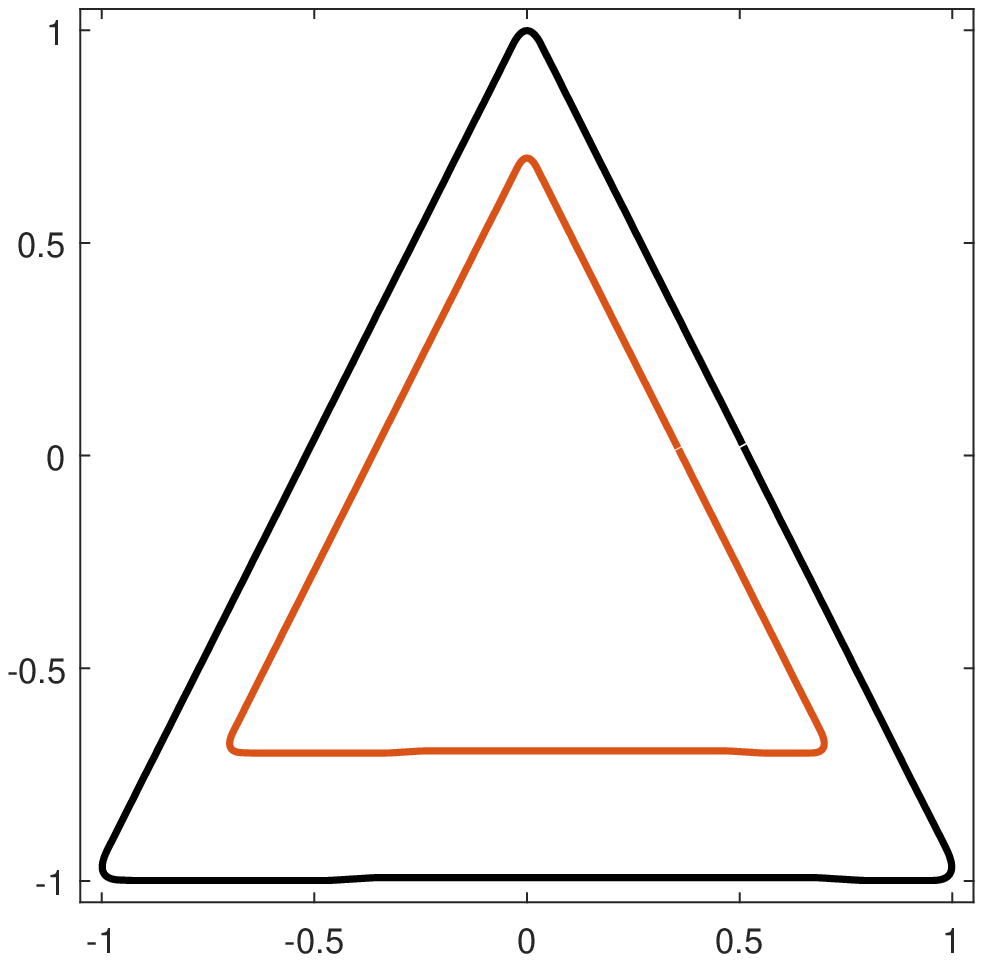}} &
   \raisebox{7ex - \height}{\includegraphics[scale=0.23]{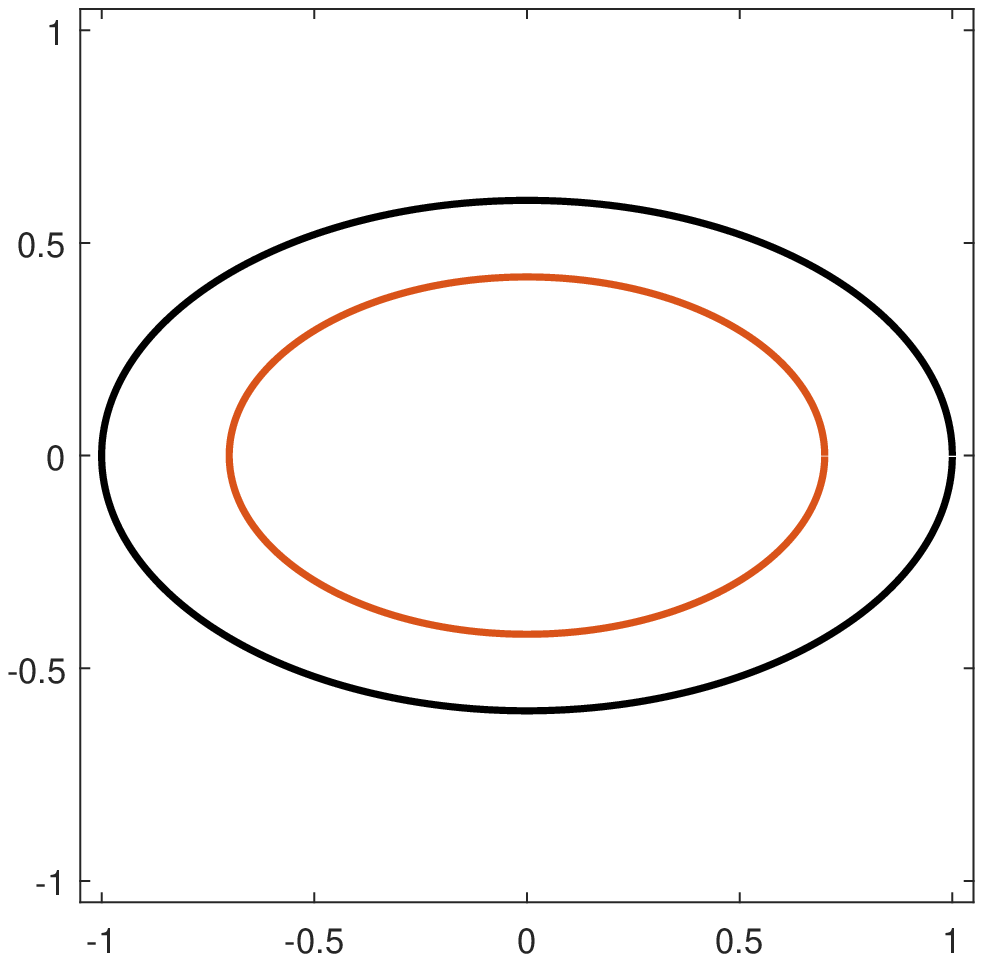}}  &
    \raisebox{7ex - \height}{\includegraphics[scale=0.23]{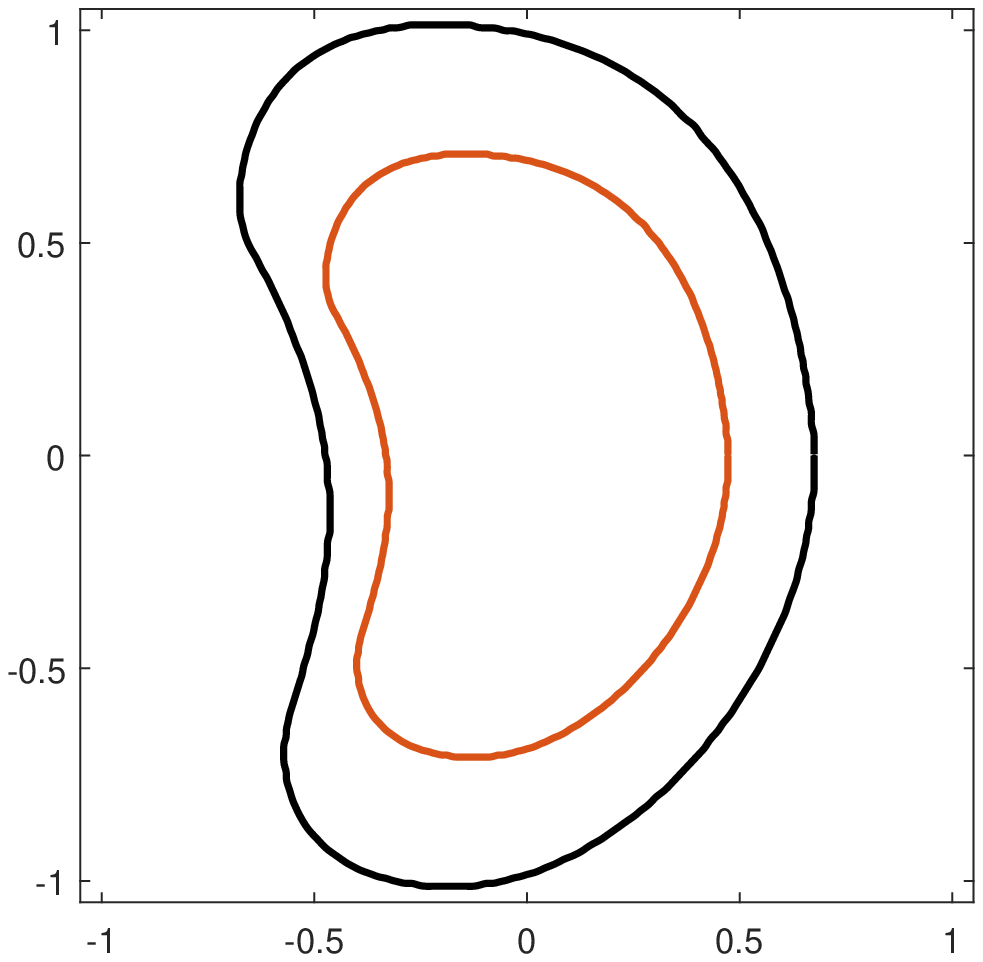}}  &
    \raisebox{7ex - \height}{\includegraphics[ scale=0.23]{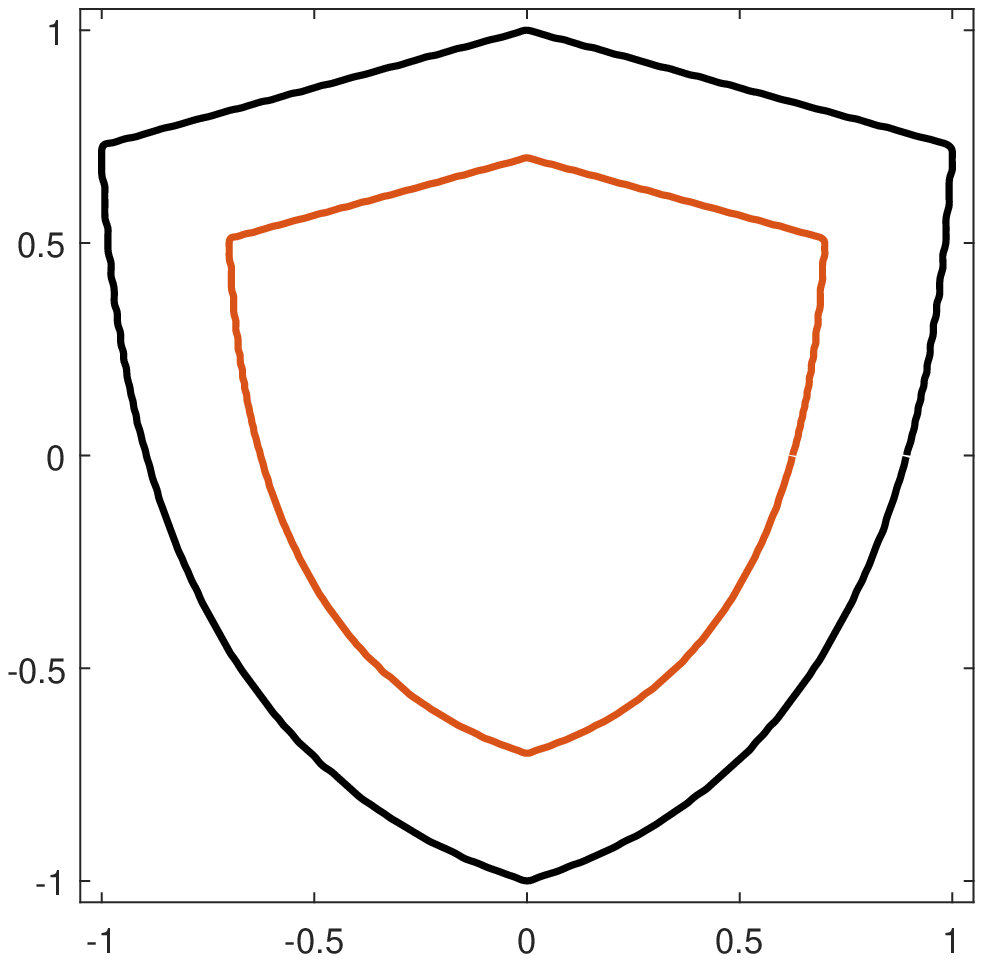}}  &  
 \raisebox{7ex - \height}{\includegraphics[ scale=0.23]{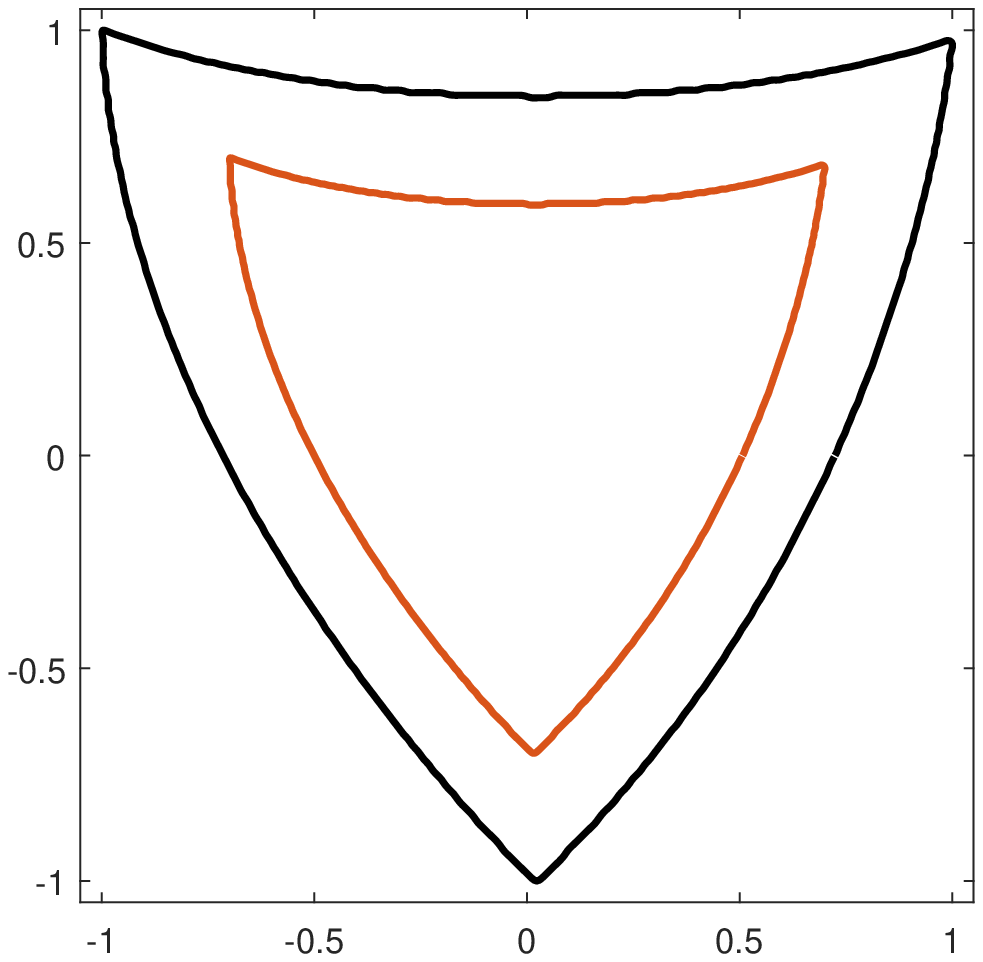}} \\
\end{tabular}
\caption{Dictionary $\mathcal{D}$.}
	\label{smalldico}
\end{figure}

 Our aim is to numerically simulate the mechanism of sensing a specific target $(\gamma , D)$ which is obtained starting from a standard element of the dictionary $(\sigma , B) \in \mathcal{D}$, applying a scaling by a parameter $s > 0$ and a rotation by an angle $\theta \in [0, 2 \pi)$. For doing so, we generalize the 
code developed in \cite{code} for homogeneous targets to piecewise inhomogeneous ones. 
\medskip

The targets we are considering for the experiments are located at the origin as the standard shapes. The scaling coefficient and the rotation angle are $s = 0.5$ and $\theta = \pi/3$, respectively. On the other hand, we consider the full-view setting. We assume that the fish is a banana-shaped fish that swims around the target along a circular trajectory whose curvature center  is located at the origin $(0,0)$ and the radius is $R = 1.5 \times \text{diameter}(D)$. We set the impedance of the skin $\xi = 0$. See for instance \figref{configuration}.

\subsection{Experiment} \label{subsec-experiment}
The experiment is as follows. As the fish swims around the target, a series of $512$ equispaced receptors on its skin collects the measurements for $500$ different positions, so that the resulting Multistatic-Response-Matrix (MSR) is a $500 \times 512$ matrix. From this acquisition procedure,  we reconstruct the CGPTs of the target up to a certain order $K$ and use a proper subset of them to compute approximately some distribution descriptors.
The descriptors obtained in this manner are then compared to the precomputed theoretical descriptors of the standard distributions of $\mathcal{D}$. We select the best matching conductivity as the standard conductivity that corresponds to the minimal error, in the noiseless case, or to the minimal mean error, when the measurements are corrupted by noise.

\begin{figure}[ht]
    \centering
        \includegraphics[height=2in]{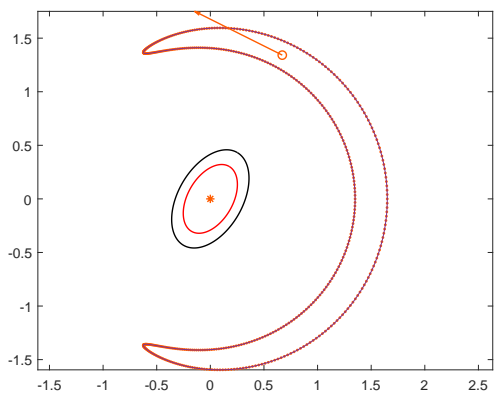}        

    \caption{Banana-shaped fish drawn at a fixed position, while swimming along a circular trajectory centered at the origin and collecting measurements for sensing the inhomogeneous target 2b.}
\label{configuration}
\end{figure}

\begin{figure}[h]
        \centering

 \begin{subfigure}[t]{0.4\textwidth}
        \centering
        \includegraphics[scale=0.5]{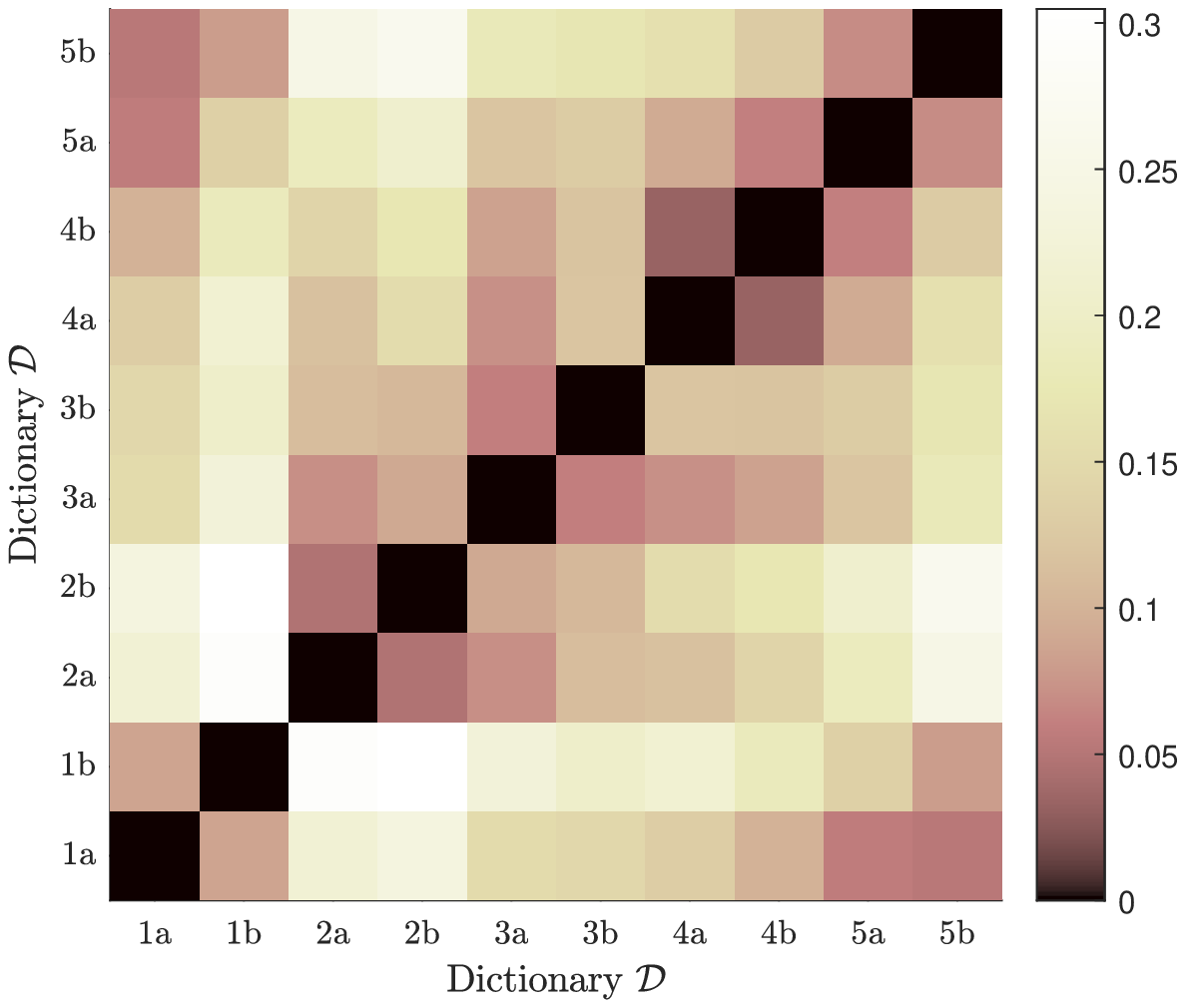}
        \caption{order $\le 2$}
\label{fig:ord2_theoret}
    \end{subfigure}
    ~ \hspace{0.05\textwidth}
 \begin{subfigure}[t]{0.4\textwidth}
        \centering
        \includegraphics[scale=0.5]{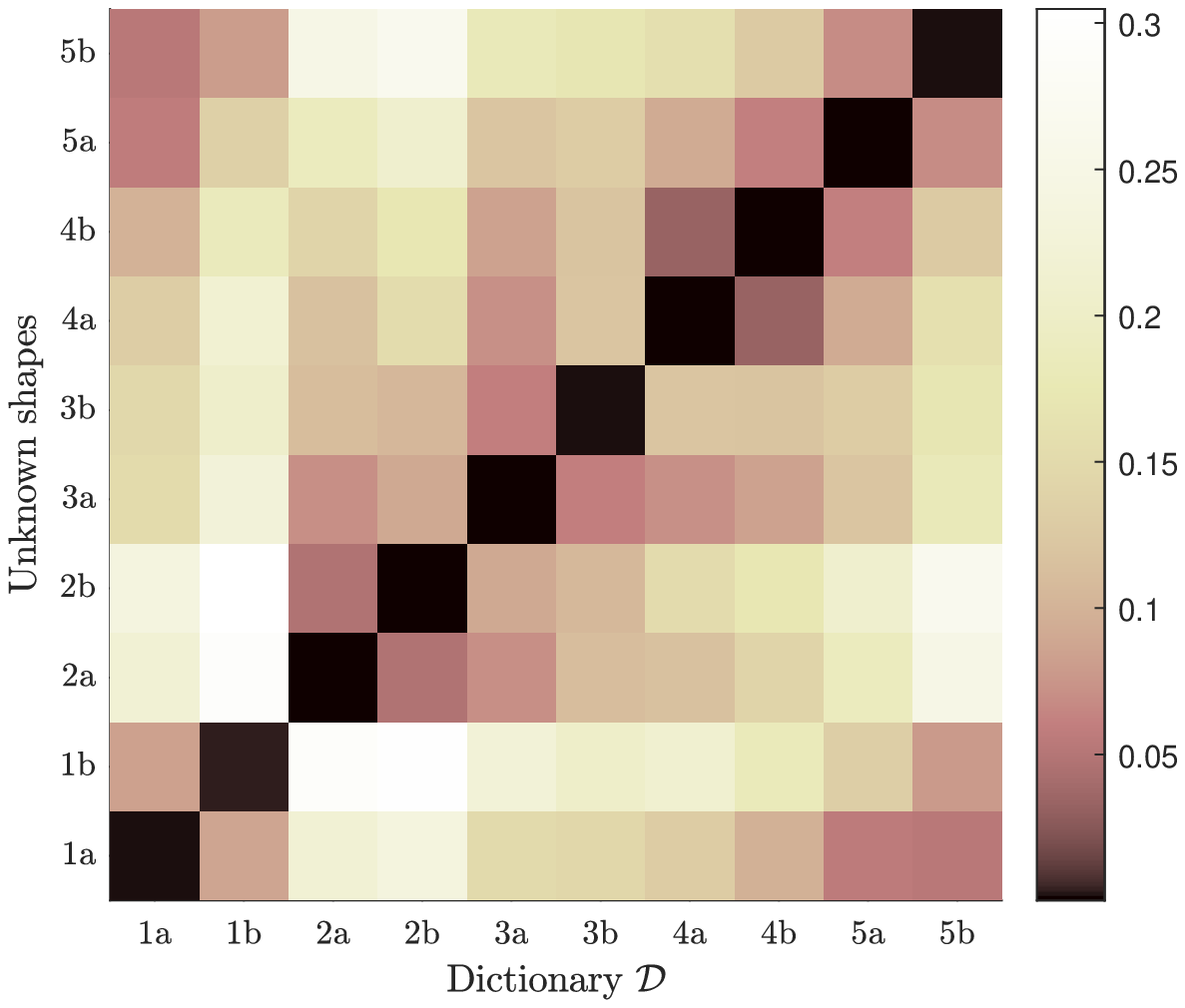}
        \caption{order $\le 2$}
\label{fig:ord2_recon}
    \end{subfigure}

 \begin{subfigure}[t]{0.4\textwidth}
        \centering
        \includegraphics[scale=0.5]{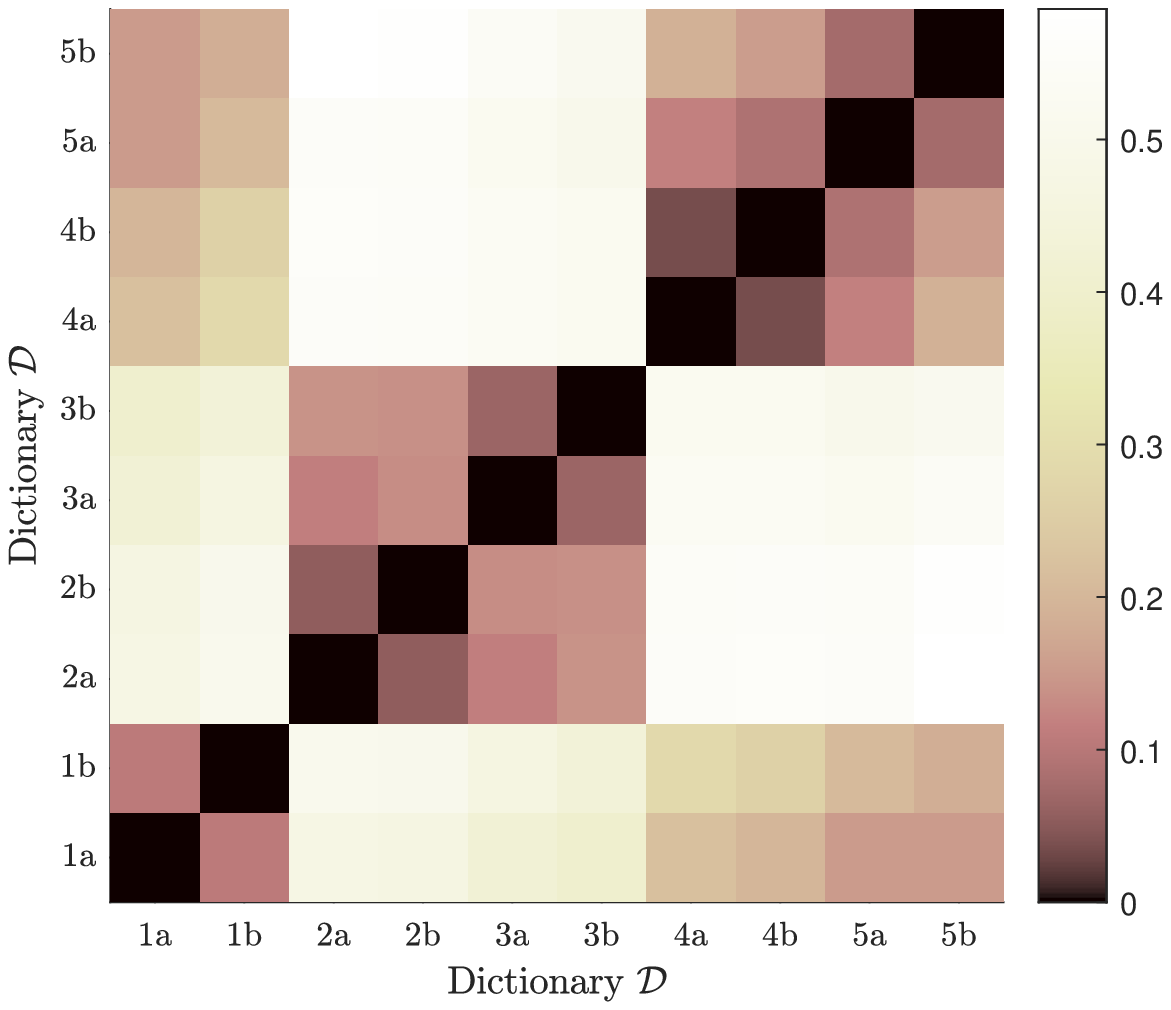}
        \caption{order $\le 3$}
\label{fig:ord3_theoret}
    \end{subfigure}
    ~ \hspace{0.05\textwidth}
 \begin{subfigure}[t]{0.4\textwidth}
        \centering
        \includegraphics[scale=0.5]{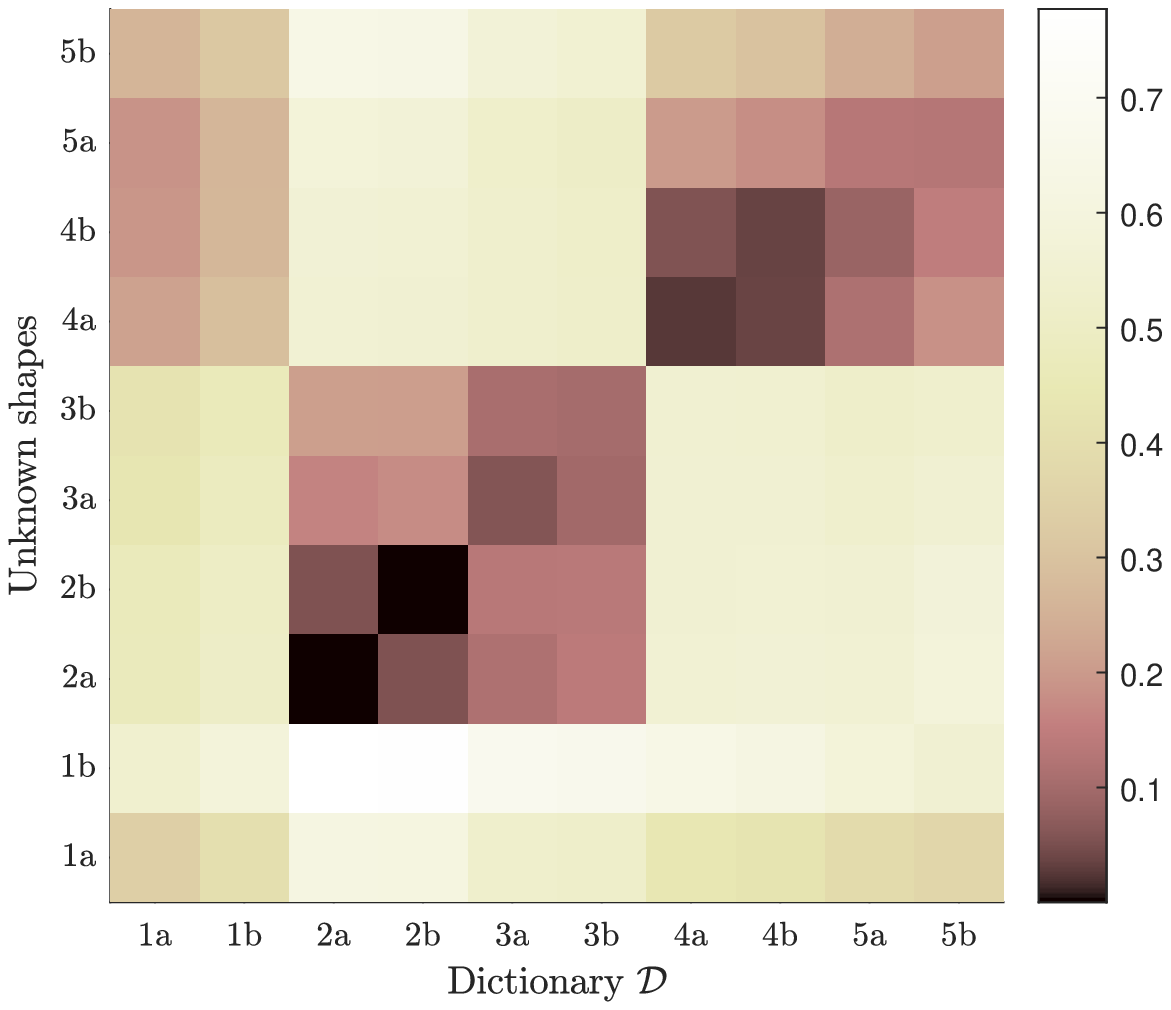}
        \caption{order $\le 3$}
\label{fig:ord3_recon}
    \end{subfigure}

\caption{\figref{fig:ord2_theoret} and \figref{fig:ord3_theoret} show the discrepancy between the theoretical  descriptors, wheras \figref{fig:ord2_recon} and \figref{fig:ord3_recon} show the discrepancy between the theoretical  descriptors and the ones obtained from the reconstructed CGPTs  at noise-level $\sigma_0 = 0$.}
	\label{small-dico-separability}
\end{figure}

We observe from \figref{small-dico-separability} that the conductivities of the dictionary $\mathcal{D}$ can be both theoretically and experimentally well-distinguished by means of their second-order descriptors $\left  ( \mathcal{I}^{(1)}_{mn} \right )_{m,n =1,2},\left  ( \mathcal{I}^{(2)}_{mn} \right )_{m,n =1,2}$.

For each noise-level, we repeat the same experiment $N = 10^4$ times and compute the probability of identification. The results are shown in \figref{prob_plot}. We report in Table \ref{frequency-table-reco} some additional data
concerning the identification that performs relatively badly, i.e., that of the target 1a.

\begin{figure}[h]
 \includegraphics[scale=0.6]{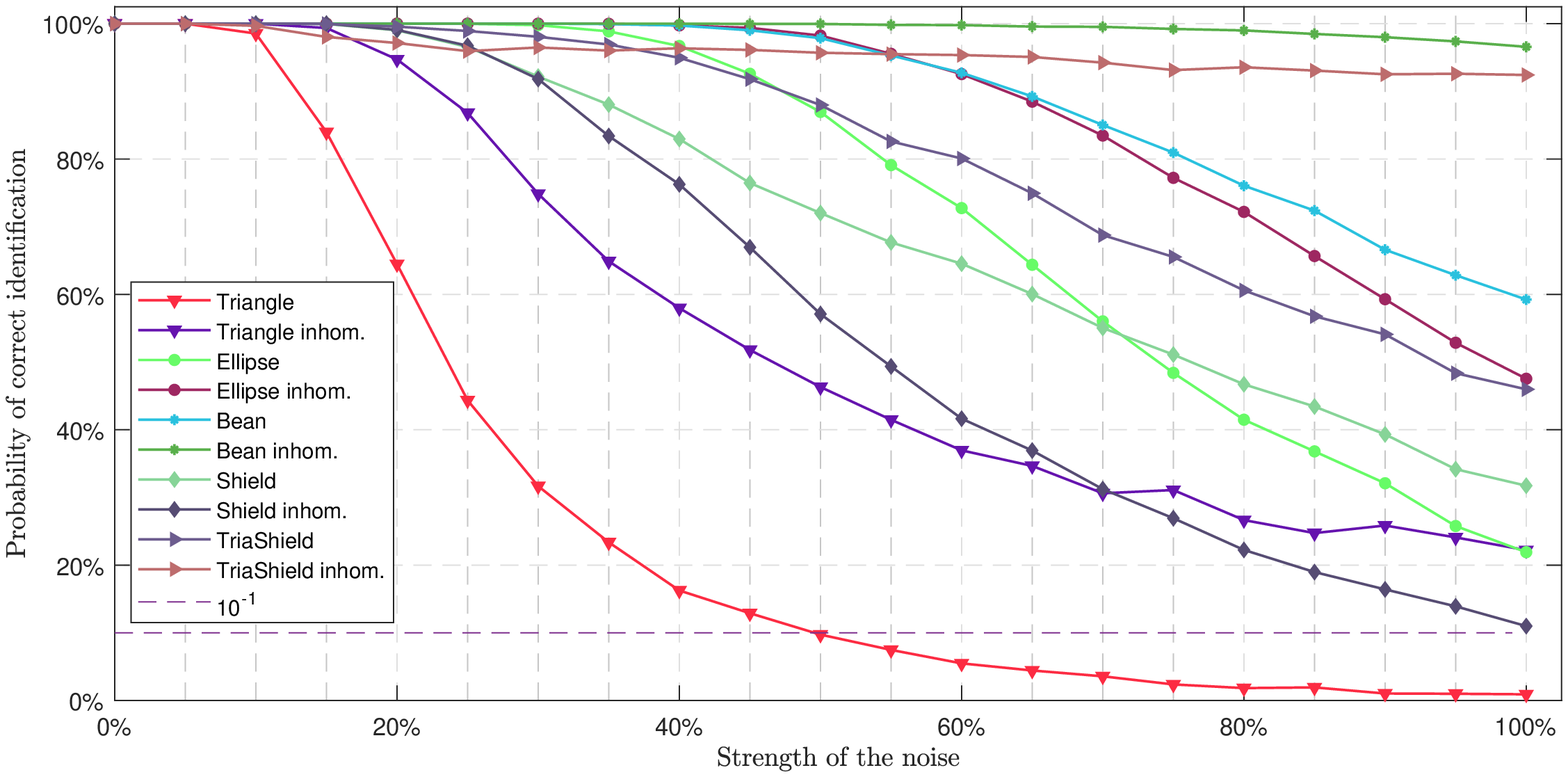} 
\caption{Stability of classification based on second-order descriptors. For each level of noise $N = 10^4$ experiments have been driven. The location of the target is supposed to be known.}
	\label{prob_plot}
\end{figure}

\begin{table}[h]
\centering
\begin{tabular}{l|cccccccccc}
&$\boxed{\text{1a}}$&1b&2a&2b&3a&3b&4a&4b&5a&5b\\
\hline
0.1&0.9854& \rule{7pt}{0pt}0& \rule{7pt}{0pt}0& \rule{7pt}{0pt}0& \rule{7pt}{0pt}0& \rule{7pt}{0pt}0& \rule{7pt}{0pt}0& \rule{7pt}{0pt}0& \rule{7pt}{0pt}0.0004& \rule{7pt}{0pt}0.0142\\ 
0.2&0.6448& \rule{7pt}{0pt}0& \rule{7pt}{0pt}0& \rule{7pt}{0pt}0& \rule{7pt}{0pt}0& \rule{7pt}{0pt}0& \rule{7pt}{0pt}0& \rule{7pt}{0pt}0& \rule{7pt}{0pt}0.0514& \rule{7pt}{0pt}0.3038\\ 
0.3&0.3168& \rule{7pt}{0pt}0& \rule{7pt}{0pt}0& \rule{7pt}{0pt}0& \rule{7pt}{0pt}0& \rule{7pt}{0pt}0& \rule{7pt}{0pt}0& \rule{7pt}{0pt}0& \rule{7pt}{0pt}0.1252& \rule{7pt}{0pt}0.558\\ 
0.4&0.1626& \rule{7pt}{0pt}0.0006& \rule{7pt}{0pt}0& \rule{7pt}{0pt}0& \rule{7pt}{0pt}0& \rule{7pt}{0pt}0& \rule{7pt}{0pt}0& \rule{7pt}{0pt}0& \rule{7pt}{0pt}0.1584& \rule{7pt}{0pt}0.6784\\ 
0.5&0.0974& \rule{7pt}{0pt}0.0012& \rule{7pt}{0pt}0& \rule{7pt}{0pt}0& \rule{7pt}{0pt}0& \rule{7pt}{0pt}0& \rule{7pt}{0pt}0& \rule{7pt}{0pt}0& \rule{7pt}{0pt}0.1712& \rule{7pt}{0pt}0.7302\\ 
\end{tabular}

\medskip

\caption{Frequency table for the identification of the conductivity 1a, i.e., the homogeneous Triangle, at different small noise-levels. Each row contains the relative frequencies for all the elements of the dictionary at a fixed noise-level.}
	\label{frequency-table-reco}
\end{table}
The results reveal that the mismatching happens more frequently between conductivities for which the corresponding geometric shapes share the same kind of high-frequency components, see \cite{Am8}. In particular, depending on the strength of the noise that is considered, the pairs of conductivities 1a-b,4a-b and 5a-b  are frequently confused with each other, due to the presence of corners and are rarely confused with the pairs 2a-b,3a-b.
This mismatching pattern is confirmed by \figref{fig:ord3_theoret}, where third-order descriptors qualitatively highlight such similarities.

\medskip

We also exhibit some plots showing the mean errors resulting from the identification procedure for two different conductivities, see \figref{fig:meanerror1a} and \figref{fig:meanerror2b}. In this case the experiment has been repeated for 5000 times, using independent draws of white noise, and the results are the mean values of all experiments.

\begin{figure}[h]
    \begin{subfigure}[t]{0.5\textwidth}
        \centering
        \includegraphics[scale=0.4]{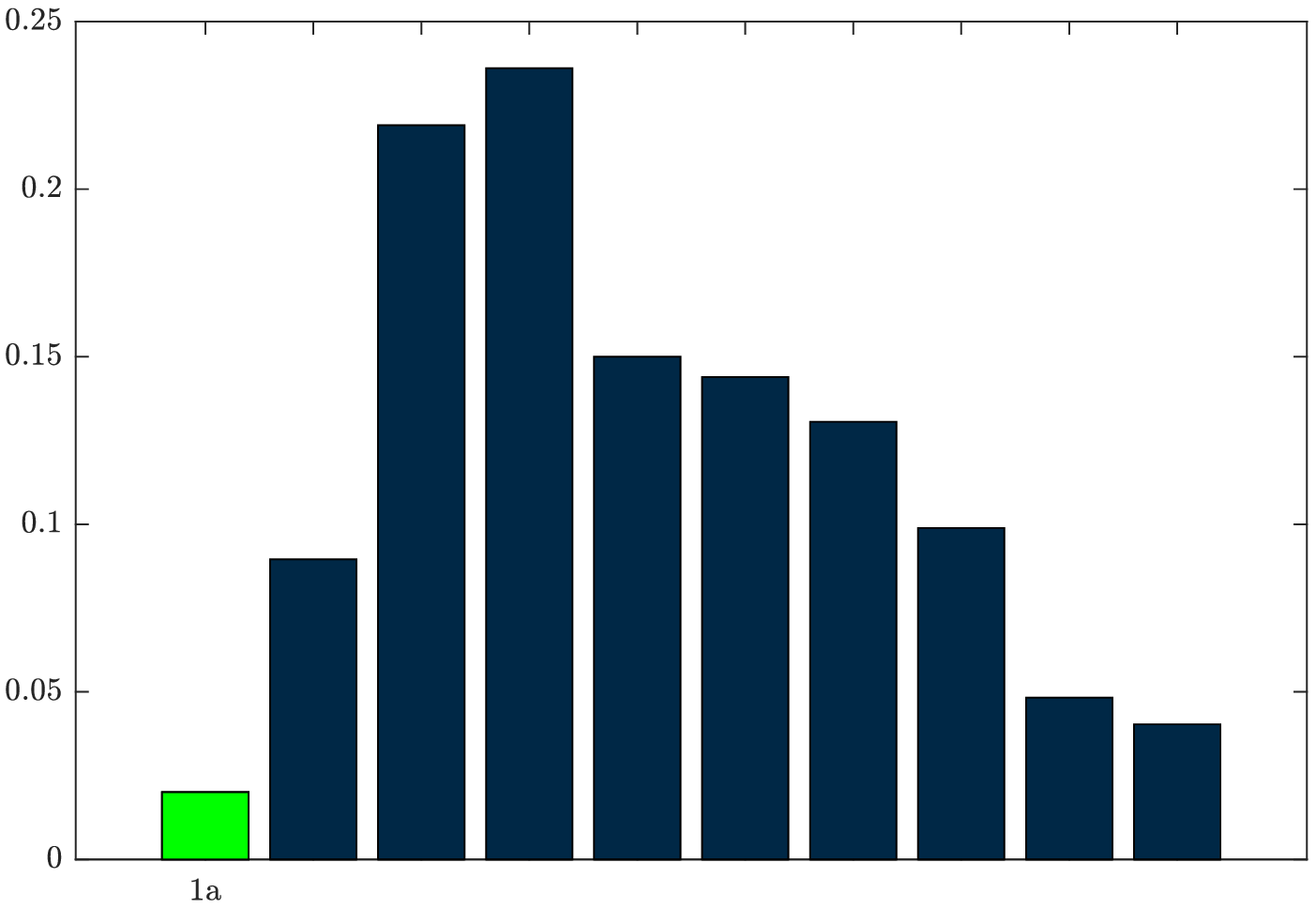}
        \caption{$\sigma_0 = 0.15$.}
    \end{subfigure}
    ~ \hspace{0.03\textwidth}
    \begin{subfigure}[t]{0.4\textwidth}
        \centering
        \includegraphics[scale=0.4]{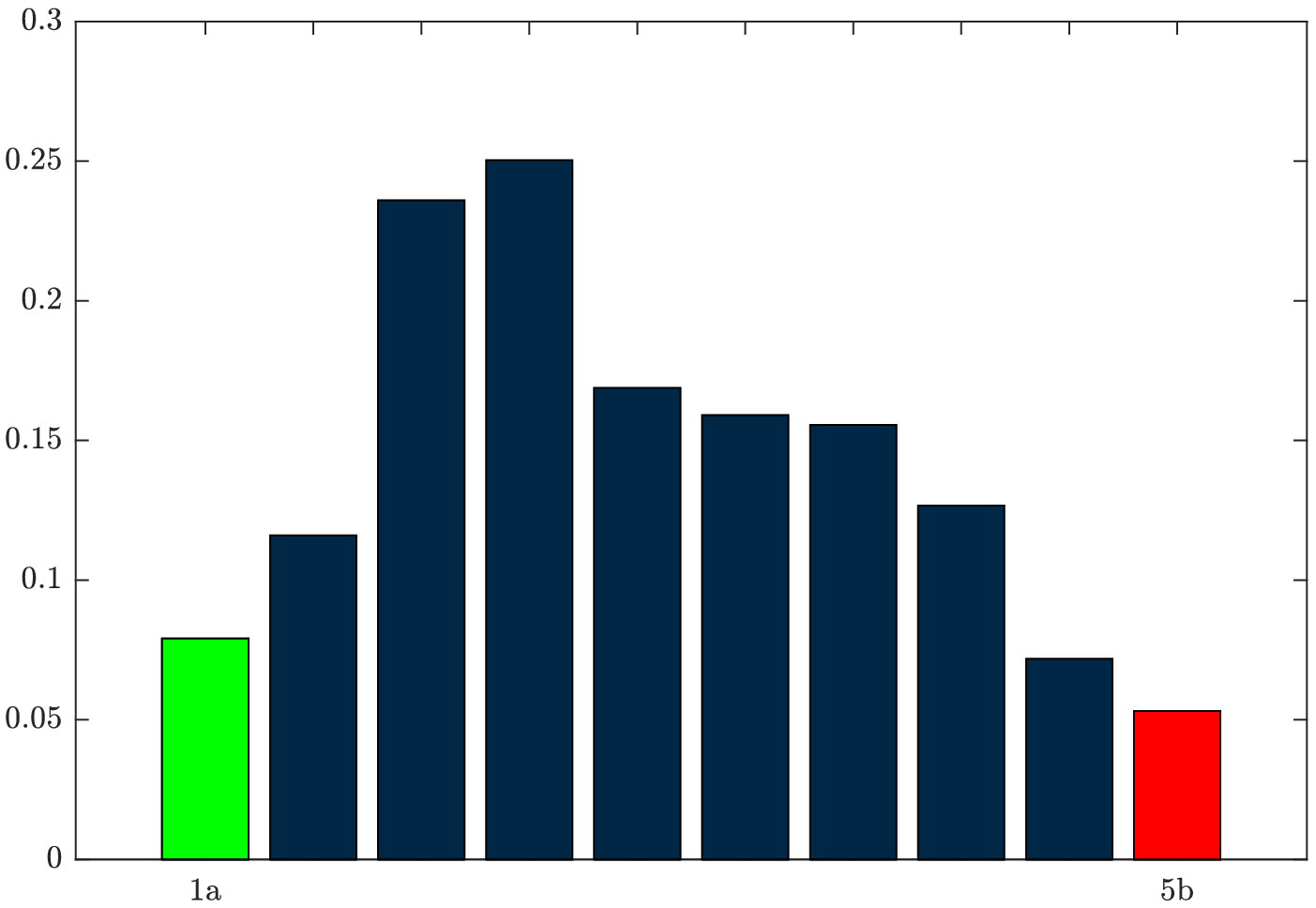}
        \caption{$\sigma_0 = 0.50$.}
    \end{subfigure}
\caption{Errors concerning the identification of the homogeneous Triangle 1a at different noise-levels. Each bar refers to a different element of $\mathcal{D}$.}
	\label{fig:meanerror1a}
\end{figure}
\begin{figure}[h]
    \begin{subfigure}[t]{0.5\textwidth}
        \centering
        \includegraphics[scale=0.4]{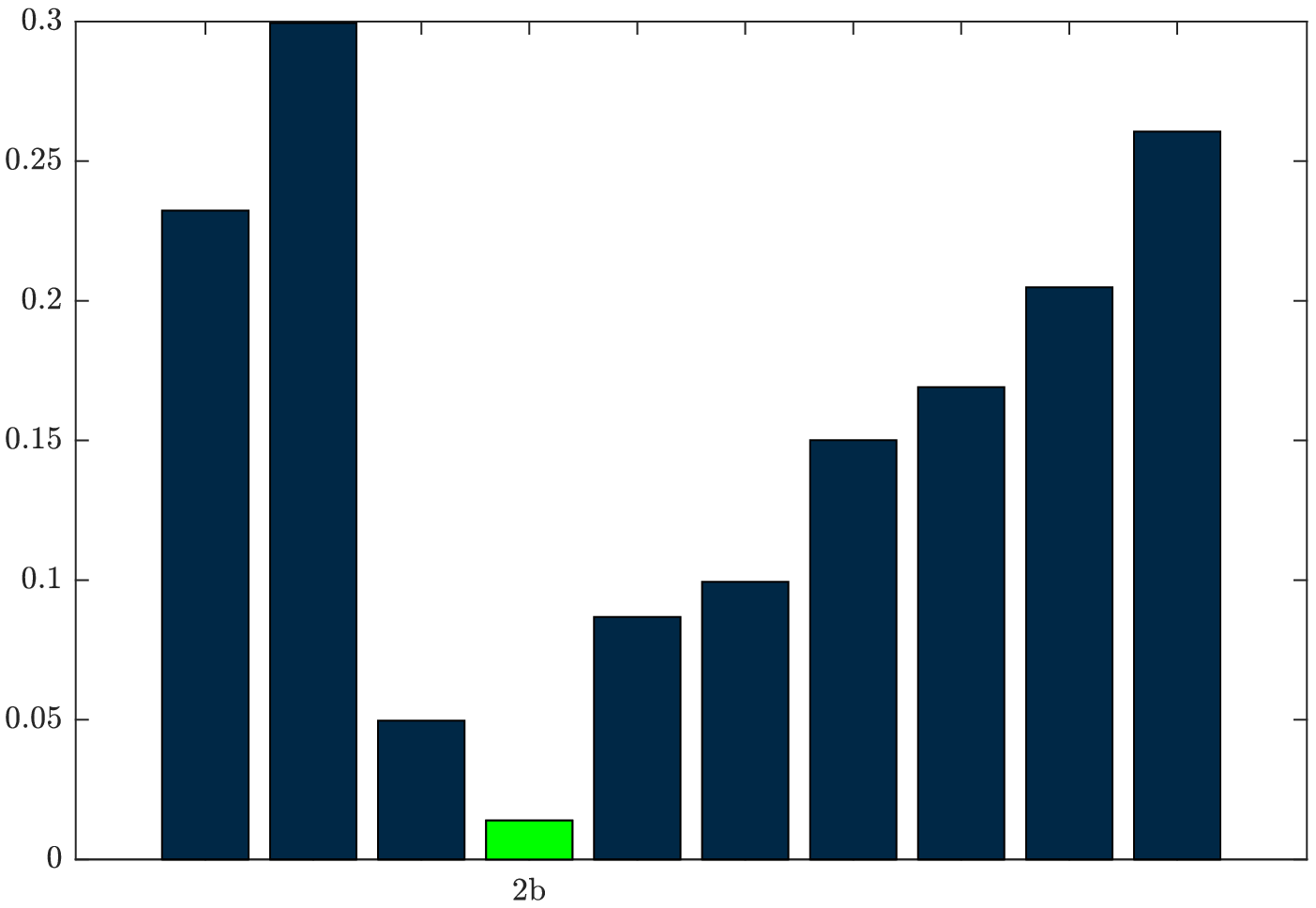}
        \caption{$\sigma_0 = 0.15$.}
    \end{subfigure}
    ~\hspace{0.03\textwidth}
    \begin{subfigure}[t]{0.4\textwidth}
        \centering
        \includegraphics[scale=0.4]{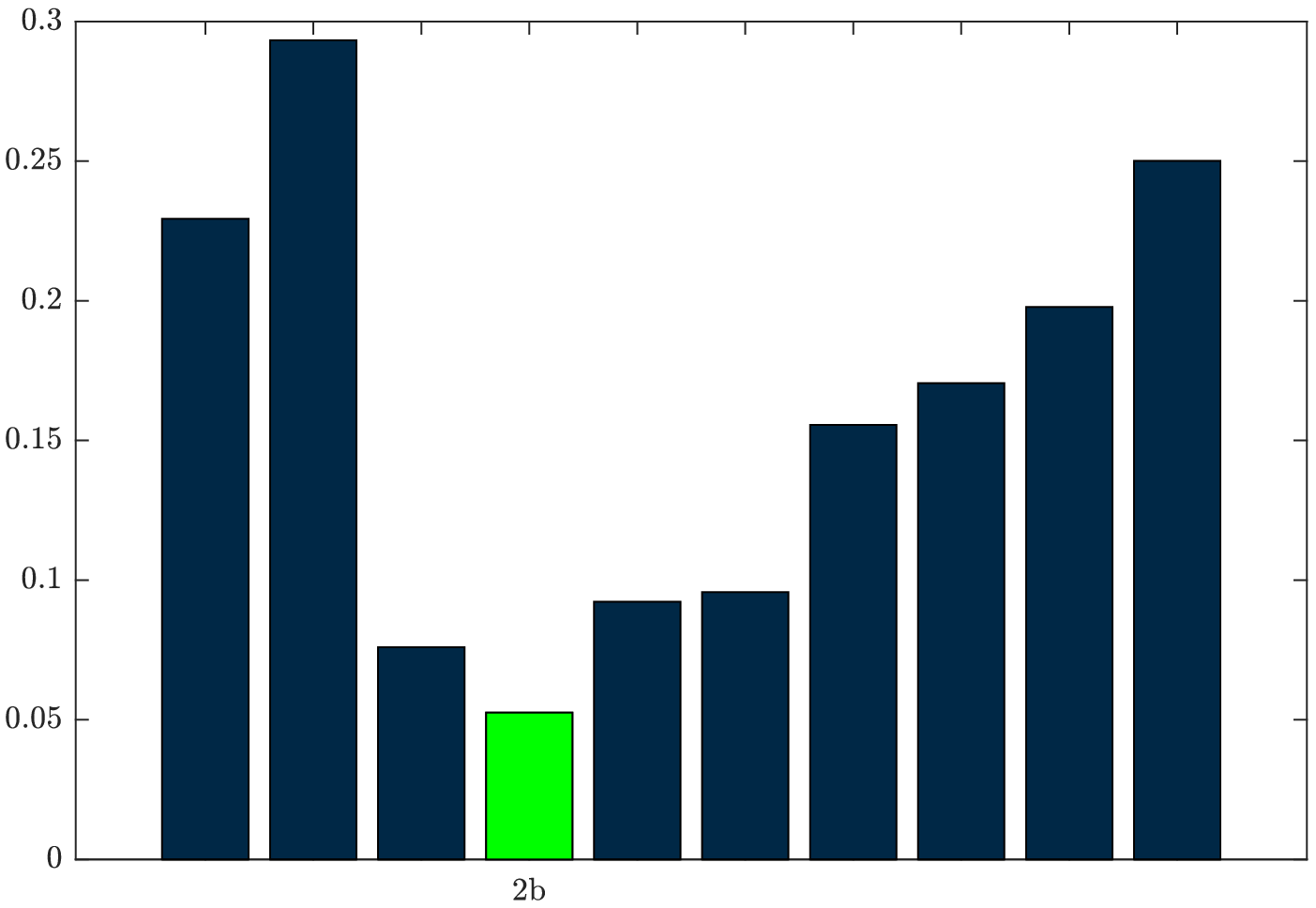}
        \caption{$\sigma_0 = 0.50$.}
    \end{subfigure}
\caption{Identification of the inhomogeneous Ellipse 2b at different noise-levels. Each bar refers to a different element of $\mathcal{D}$.}
	\label{fig:meanerror2b}
\end{figure}

\FloatBarrier
\subsection{Robustness of the reconstruction}

We numerically reconstructed the CGPTs from the measurements, i.e., from the MSR matrix. This reconstruction turns out to be robust when we add some noise to the simulated data. Fixing the truncation order in the reconstruction at $K = 5$, the relative error of the reconstructed CGPTs of orders $k$ for  $k \le 5$ is illustrated in Figure \ref{robustness}. For the noisy case, the experiment has been repeated $100$ times, using independent draws of white noise and the reconstructed CGPT is taken as the average of the CGPTs.

\begin{figure}[ht]
    \begin{subfigure}[t]{0.5\textwidth}
        \centering
        \includegraphics[scale=0.5]{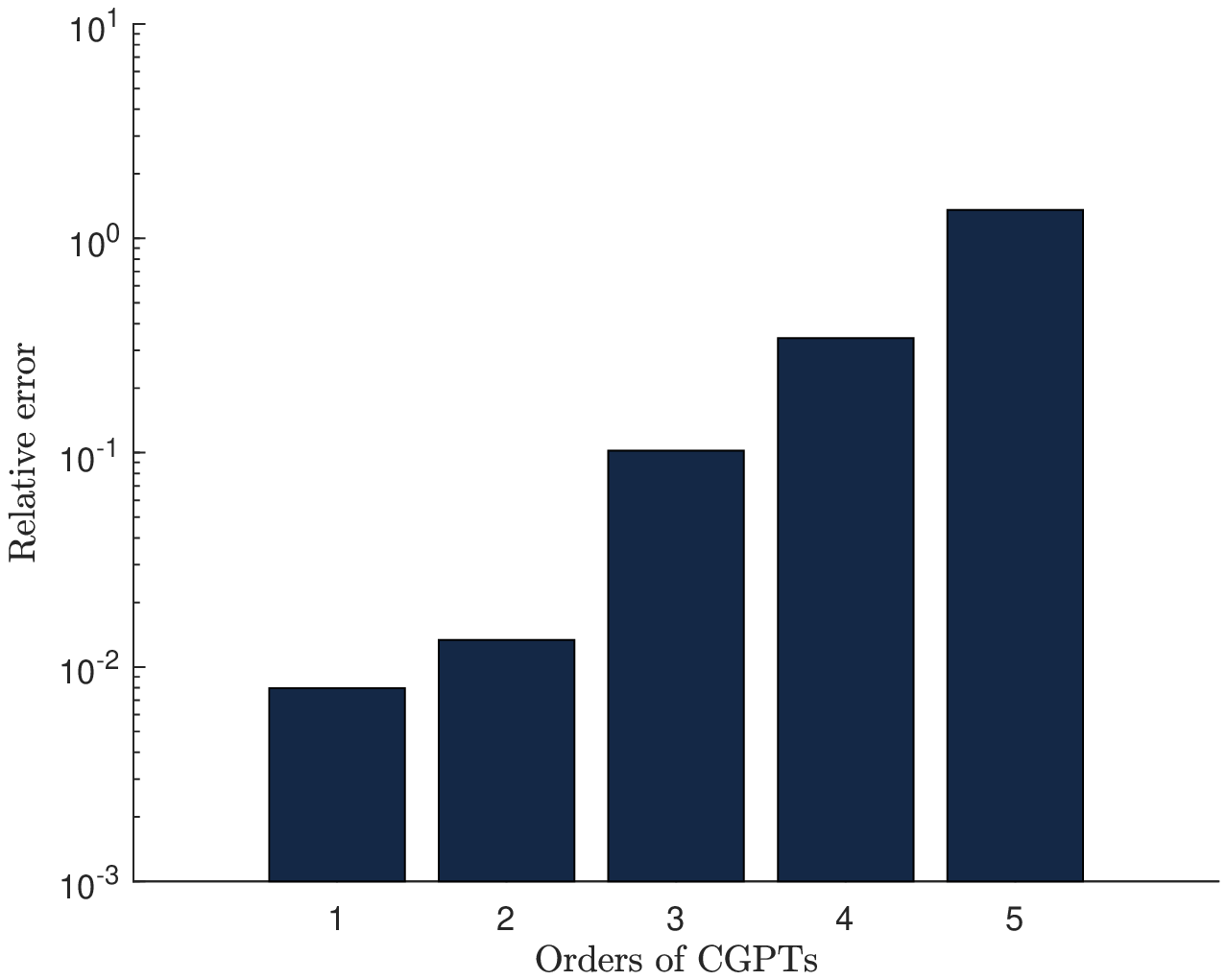}
        \caption{$\sigma_0 = 0$.}
    \end{subfigure}
    ~ 
    \begin{subfigure}[t]{0.5\textwidth}
        \centering
        \includegraphics[scale=0.5]{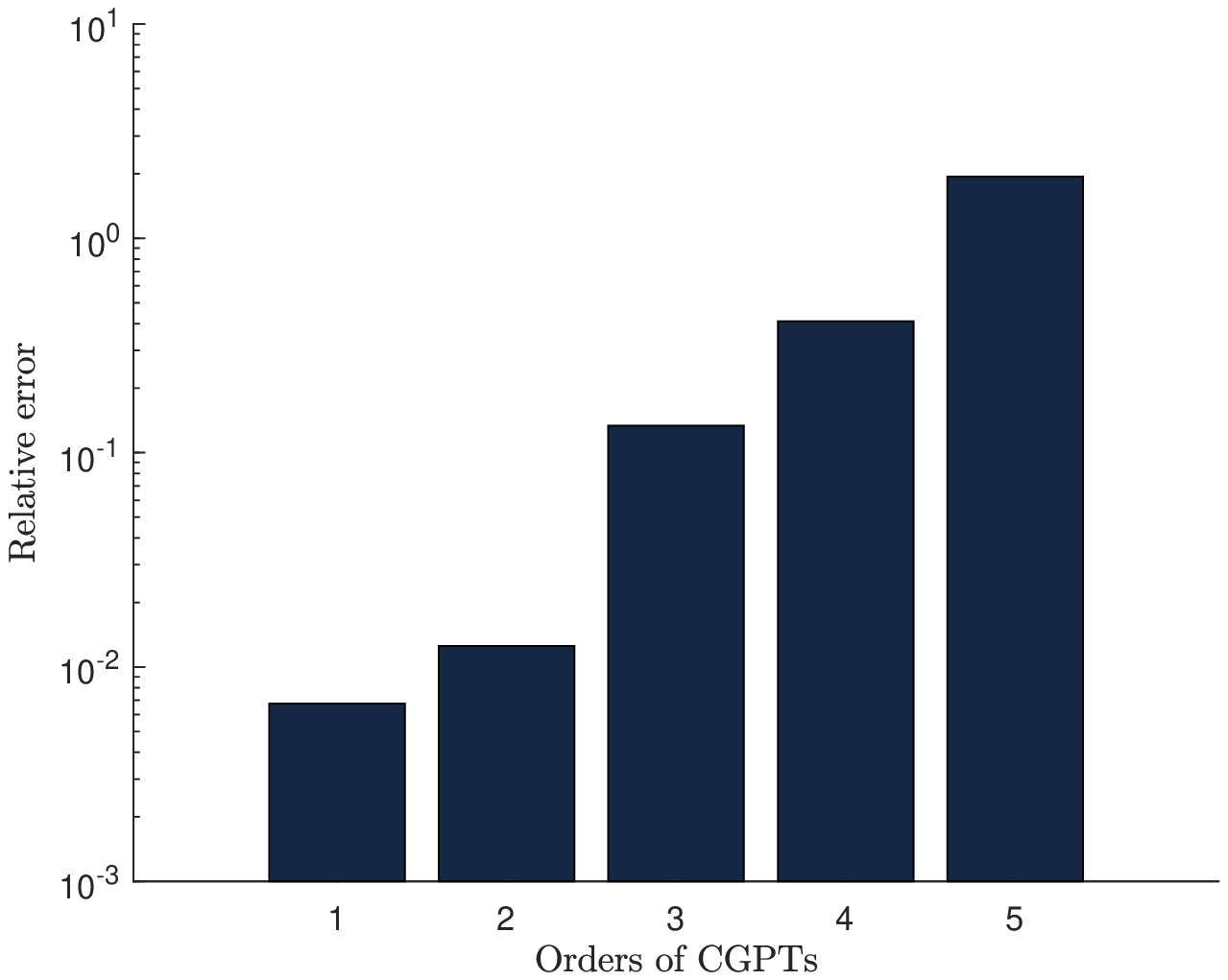}
        \caption{$\sigma_0 = 0.2$.}
    \end{subfigure}
\caption{Relative error $\| \mathbf{M} - \mathbf{M}_{\text{recon}} \|_F/\|\mathbf{M}\|_F$  of the reconstruction of the CGPTs throughout the acquisition procedure described previously in \ref{subsec-setting} for the conductivity 1b, the inhomogeneous Triangle.}
	\label{robustness}
\end{figure}

\FloatBarrier

\section{Concluding remarks} 
In this paper, we have extended the dictionary-matching approach for classification in electro-sensing to inhomogeneous targets. We have established translation, rotation, and scaling formulas for particular linear combinations of the generalized polarization tensors associated with inhomogeneous targets. We have derived new invariants and tested their performance for recognizing inhomogeneous targets inside a dictionary of homogeneous and inhomogeneous conductivity distributions. In a forthcoming paper, we plan to combine our present approach together with the multi-frequency approach introduced in \cite{Am5} to enhance the classification capabilities of the proposed method and its stability.

\section{Acknowledgment}
 The author gratefully acknowledges Prof. H. Ammari for his guidance and the financial support granted by the Swiss National Foundation (grant 200021-172483).

\end{document}